\documentclass[leqno,10.5pt]{article} 
\usepackage{amsmath, amssymb}   
\usepackage{amsthm} 
\usepackage{amscd}    
\usepackage{amsmath, amssymb}  
\usepackage{amsthm} 
\usepackage{amscd}   
\usepackage{color}  
\theoremstyle{plain} 
\newtheorem{theorem}{\indent\sc Theorem}[section]
\newtheorem{lemma}[theorem]{\indent\sc Lemma}
\newtheorem{corollary}[theorem]{\indent\sc Corollary}
\newtheorem{proposition}[theorem]{\indent\sc Proposition}

\theoremstyle{definition} 
\newtheorem{definition}[theorem]{\indent\sc Definition}
\newtheorem{remark}[theorem]{\indent\sc Remark}

\newcommand{\C}{\mathbb{C}}
\newcommand{\R}{\mathbb{R}}
\newcommand{\Q}{\mathbb{Q}}
\newcommand{\Z}{\mathbb{Z}}
\newcommand{\N}{\mathbb{N}}

\title{Pad\'{e} approximations for products of functions}
\author{\textsc{Makoto Kawashima}}
\date{ } 
\begin{document}

\maketitle  

 
\begin{abstract}
In this article, we construct new Pad\'{e} approximations for the \emph{product} of binomial functions and powers of logarithmic functions. While several explicit Pad\'{e} approximants are known for powers of exponential functions, binomial functions, and logarithmic functions individually, an explicit Pad\'{e} construction for the product of these functions has not yet been directly achieved.
Our main result yields arithmetic applications, providing new linear independence measures for linear forms in $(1+\alpha)^{\omega_i}\log^{j_i}(1+\alpha)$ for $1 \le i \le m$ and $0 \le j_i \le r_i - 1$, where $0 < m, r_1, \ldots, r_m \in \mathbb{Z}_{\geq 1}$, $\omega_1, \ldots, \omega_m \in \mathbb{Q}$, and $0 \le \omega_1 < \cdots < \omega_m < 1$. These results hold with algebraic coefficients in both the complex and $p$-adic cases. Additionally, we establish that Pad\'{e} approximation of a single polylogarithm is, in general, perfect.
\end{abstract}
\textit{Key words and phrases}:~Pad\'e approximation, product of functions, perfectness, logarithm, binomial function and polylogarithm.

\section{Introduction}
As is well-known, Hermite-Pad\'{e} approximation is a central tool in Diophantine problems, providing explicit and effective results on transcendence \cite{FN}, as well as irrationality \cite{S1}. 
Despite its extensive application, explicit Hermite-Pad\'{e} approximations have been relatively scarce, particularly for functions expressed as products of distinct functions.

In this article, we explicitly construct Pad\'{e} approximations for the \textit{product} of two distinct hypergeometric functions, namely, binomial functions and powers of logarithmic functions. K.~Mahler's contributions include Pad\'{e} approximations for powers of exponential functions \cite{M0, M1}, binomial functions \cite{M}, and powers of logarithmic functions \cite{M2}, all applied to Diophantine problems. Subsequent refinement of Mahler's works has been undertaken by various mathematicians, notably G.~V.~Chudnovsky for binomial functions \cite{ch1, ch9}, and for powers of logarithmic functions by Chudnovsky \cite{chpi} and E.~Reyssat \cite{R}. M. Huttner \cite{Huttner} also pursued Pad\'{e} constructions for specific products through the monodromy method, though without direct arithmetic applications.

Our method takes a completely formal approach, treating every power series formally, which distinguishes it from previous works. While the basic structure of the proof is classical, due to Ch.~Hermite, H.~E.~Pad\'{e}, and C.~L.~Siegel \cite{S}, our new ingredient lies in the explicit construction of Pad\'{e} approximations for the product of binomial functions and powers of logarithm functions, invoking a linear independence measure of associated numbers. To achieve this construction, the key point is an entire transformation of Pad\'{e} approximations of exponential functions.

An advantage of our approach, thanks to the perfectness of exponential functions in the sense of Mahler ({\it confer} \cite{M4}), is that we obtain the normality of certain indices of Pad\'{e} approximants for the product of binomial functions and powers of logarithm functions. This, in turn, leads to the regularity of the matrices associated with Pad\'{e} approximants (see Corollary $\ref{key corollary}$), which is a crucial step in completing the proof.

\section{Notation and Statement of the Main Result}
We collect here the notation used throughout this article.

We denote by $\mathbb{N}$ the set of strictly positive integers.  
Let $K$ be an algebraic number field. We denote by $\mathfrak{M}_K$ the set of all places of $K$, by $\mathfrak{M}_K^{\infty}$ the set of infinite places, and by $\mathfrak{M}_K^f$ the set of finite places.

For $v \in \mathfrak{M}_K$, let $K_v$ denote the completion of $K$ at $v$.  
We define the normalized absolute value $|\cdot|_v$ on $K$ as follows:
\begin{align*}
|p|_v &= p^{-\tfrac{[K_v:\mathbb{Q}_p]}{[K:\mathbb{Q}]}} 
\quad \text{if } v \in \mathfrak{M}_K^f \text{ and } v \mid p, \\
|x|_v &= |\iota_v x|^{\tfrac{[K_v:\mathbb{R}]}{[K:\mathbb{Q}]}} 
\quad \text{if } v \in \mathfrak{M}_K^{\infty},
\end{align*}
where $p$ is a rational prime and $\iota_v: K \hookrightarrow \mathbb{C}$ is the embedding corresponding to $v$.

Let $m$ be a non-negative integer and $\boldsymbol{\beta} = (\beta_0, \ldots, \beta_m) \in K^{m+1} \setminus \{ \mathbf{0} \}$.  
We define the (absolute) multiplicative height of $\boldsymbol{\beta}$ by
\[
\mathrm{H}(\boldsymbol{\beta}) = \prod_{v \in \mathfrak{M}_K} \max\{1, |\beta_0|_v, \ldots, |\beta_m|_v\},
\]
and the logarithmic absolute Weil height by
\[
\mathrm{h}(\boldsymbol{\beta}) = \log \mathrm{H}(\boldsymbol{\beta}).
\]
We write
\[
\mathrm{h}_v(\boldsymbol{\beta}) := \log \max\{1, |\beta_0|_v, \ldots, |\beta_m|_v\},
\]
so that
\[
\mathrm{h}(\boldsymbol{\beta}) = \sum_{v \in \mathfrak{M}_K} \mathrm{h}_v(\boldsymbol{\beta}).
\]

For a finite set $S$ of algebraic numbers, we define its denominator by
\[
\mathrm{den}(S) := \min \left\{ n \in \mathbb{N} \mid n\alpha \text{ is an algebraic integer for all } \alpha \in S \right\}.
\]

\bigskip

We now introduce further notation needed to state our main result.
Let $m, r_1, \ldots, r_m \in \mathbb{N}$, let $K$ be an algebraic number field, and let $v_0$ be a place of $K$. Define $r := \max\{r_1, \ldots, r_m\}$, and set
\[
\varepsilon_{v_0} := 
\begin{cases}
1 & \text{if } v_0 \in \mathfrak{M}_K^{\infty}, \\
0 & \text{if } v_0 \in \mathfrak{M}_K^f.
\end{cases}
\]
For $v_0 \in \mathfrak{M}_K^f$, let $p_{v_0}$ denote the rational prime below $v_0$.

Let $\omega_1, \ldots, \omega_m \in \mathbb{Q}$ with $0 \leq \omega_1 < \cdots < \omega_m < 1$.  
Define the following real-valued functions:
\begin{align*}
\mathbb{A}_{v_0}(\boldsymbol{\omega}, \alpha) &= 
\left(\sum_{i=1}^m r_i - 1 \right) \mathrm{h}_{v_0}(\alpha)
- \frac{\varepsilon_{v_0} [K_{v_0} : \mathbb{R}]}{[K:\mathbb{Q}]} \, \mathrm{C}(\boldsymbol{\omega}) \\
&\qquad - (\varepsilon_{v_0} - 1)
\begin{cases}
0 & \text{if } p_{v_0} \nmid \mathrm{den}(\omega_i)_{1 \leq i \leq m}, \\
\displaystyle \frac{\left(\sum_{i=1}^m r_i\right)p_{v_0}}{p_{v_0} - 1} \log |p_{v_0}|_{v_0} & \text{otherwise},
\end{cases} \\
\mathbb{B}_{v_0}(\boldsymbol{\omega}, \alpha) &= 
\left(\sum_{i=1}^m r_i - 1 \right)\left( \mathrm{h}(\alpha) + \mathrm{C}(\boldsymbol{\omega}) \right)
- \mathrm{h}_{v_0}(\alpha)
- \frac{\varepsilon_{v_0} [K_{v_0} : \mathbb{R}]}{[K:\mathbb{Q}]} \, \mathrm{C}(\boldsymbol{\omega}), \\
U_{v_0}(\boldsymbol{\omega}, \alpha) &= \mathrm{h}_{v_0}(\alpha) + \frac{\varepsilon_{v_0} [K_{v_0} : \mathbb{R}]}{[K:\mathbb{Q}]} \, \mathrm{C}(\boldsymbol{\omega}), \\
V_{v_0}(\boldsymbol{\omega}, \alpha) &= \mathbb{A}_{v_0}(\boldsymbol{\omega}, \alpha) - \mathbb{B}_{v_0}(\boldsymbol{\omega}, \alpha),
\end{align*}
where $\varphi$ is Euler's totient function, and
\[
\mathrm{C}(\boldsymbol{\omega}) :=
\left( \sum_{i=1}^m r_i + 1 \right)\log 2 
+ 3r \sum_{1 \leq i_1 < i_2 \leq m} \left(\dfrac{{\rm{den}}(\omega_{i_2}-\omega_{i_1})}{\varphi({\rm{den}}(\omega_{i_2}-\omega_{i_1}))}\sum_{\substack{1\le j \le {\rm{den}}(\omega_{i_1}-\omega_{i_2}) \\ (j,{\rm{den}}(\omega_{i_2}-\omega_{i_1}))=1}}\dfrac{1}{j}\right)+ r.
\]

\bigskip

We can now state our main theorem.

\begin{theorem}\label{power of log indep}
Assume
\begin{align*}
&|\alpha|_{v_0} >
\begin{cases}
1 & \text{if } v_0 \in \mathfrak{M}_K^{\infty} \text{ or } v_0 \in \mathfrak{M}_K^f \text{ and } p_{v_0} \nmid \mathrm{den}(\omega_1, \ldots, \omega_m), \\
|p_{v_0}|_{v_0}^{-p_{v_0}/(p_{v_0} - 1)} & \text{otherwise},
\end{cases}\\
&V_{v_0}(\boldsymbol{\omega}, \alpha) > 0.
\end{align*}
Then for any $0 < \varepsilon < V_{v_0}(\boldsymbol{\omega}, \alpha)$, there exists an effectively computable constant $H_0 = H_0(\varepsilon)$ depending only on $\varepsilon$ and the given data such that the following holds.

For any $\boldsymbol{\beta} = (\beta_{i,j_i})_{1 \leq i \leq m,\ 0 \leq j_i \leq r_i - 1} \in K^{r_1 + \cdots + r_m} \setminus \{\mathbf{0}\}$ with $\mathrm{H}(\boldsymbol{\beta}) \geq H_0$, we have
\[
\left| \sum_{i=1}^m \sum_{j_i=0}^{r_i - 1} \beta_{i,j_i} (1+\alpha^{-1})^{\omega_i} \log^{j_i}(1+\alpha^{-1}) \right|_{v_0}
> C(\boldsymbol{\omega}, \alpha, \varepsilon) \cdot H_{v_0}(\boldsymbol{\beta}) \cdot \mathrm{H}(\boldsymbol{\beta})^{-\mu(\boldsymbol{\omega}, \alpha, \varepsilon)},
\]
where
\begin{align*}
\mu(\boldsymbol{\omega}, \alpha, \varepsilon) &=
\frac{\mathbb{A}_{v_0}(\boldsymbol{\omega}, \alpha) + U_{v_0}(\boldsymbol{\omega}, \alpha)}{V_{v_0}(\boldsymbol{\omega}, \alpha) - \varepsilon}, \\
C(\boldsymbol{\omega}, \alpha, \varepsilon) &=
\frac{1}{2} \exp\left[
- \frac{
\left(V_{v_0}(\boldsymbol{\omega}, \alpha) - \varepsilon + \log 2 \right)
\left( \mathbb{B}_{v_0}(\boldsymbol{\omega}, \alpha) + U_{v_0}(\boldsymbol{\omega}, \alpha) + \varepsilon \right)
}{
V_{v_0}(\boldsymbol{\omega}, \alpha) - \varepsilon
}
\right].
\end{align*}
\end{theorem}

\begin{remark}
When $\alpha$ varies, the $\alpha$-dependent part of $V_{v_0}(\boldsymbol{\omega}, \alpha)$ is
\[
\mathrm{h}_{v_0}(\alpha) + \left(\sum_{i=1}^m r_i - 1\right) \left( \mathrm{h}_{v_0}(\alpha) - \mathrm{h}(\alpha) \right)
= \mathrm{h}_{v_0}(\alpha) - \left(\sum_{i=1}^m r_i - 1\right) \sum_{v \neq v_0} \mathrm{h}_v(\alpha),
\]
while the other terms are independent of $\alpha$.
Therefore, if this quantity is sufficiently large, the condition $V_{v_0}(\boldsymbol{\omega}, \alpha) > 0$ is fulfilled.
In particular, for any number field $K$ and place $v_0$ of $K$, the strong approximation theorem ensures the existence of infinitely many $\alpha \in K$ with $V_{v_0}(\boldsymbol{\omega}, \alpha) > 0$.
\end{remark}

\bigskip

In the special case $K = \mathbb{Q}$, $m = 1$, $\omega_1 = 0$, and $v_0 = \infty$, we obtain the following corollary.

\begin{corollary}\label{corollary main theorem}
Let $r \geq 2$ be an integer, and let $\alpha \in \mathbb{Q} \setminus \{0, -1\}$ with $|\alpha| > 1$. Define
\[
V(\alpha) := r \mathrm{h}_\infty(\alpha) - (r-1)\left( \mathrm{h}(\alpha) + (r+1)\log 2 + r \right).
\]
Assume $V(\alpha) > 0$. Then for any $0 < \varepsilon < V(\alpha)$, there exists an effectively computable constant $H_0 = H_0(\alpha, \varepsilon)$ such that for any $\boldsymbol{b} = (b_j)_{0 \leq j \leq r-1} \in \mathbb{Z}^r \setminus \{\mathbf{0}\}$ with $\mathrm{H}(\boldsymbol{b}) \geq H_0$, we have
\[
\left| \sum_{j=0}^{r-1} b_j \log^j(1+\alpha^{-1}) \right|
> C(\alpha, \varepsilon) \cdot \mathrm{H}(\boldsymbol{b})^{1 - \mu(\alpha, \varepsilon)},
\]
where
\begin{align*}
\mu(\alpha, \varepsilon) &= \frac{r \mathrm{h}_\infty(\alpha)}{V(\alpha) - \varepsilon}, \\
C(\alpha, \varepsilon) &= \frac{1}{2} \exp\left[
- \frac{
(V(\alpha) - \varepsilon + \log 2)
((r-1)(\mathrm{h}(\alpha) + (r+1)\log 2 + r) + \varepsilon)
}{V(\alpha) - \varepsilon}
\right].
\end{align*}
\end{corollary}

\bigskip

The following table provides examples of integers $r$ and $\alpha \in \mathbb{Z}$ for which $V(\alpha) > 0$:

\[
\begin{array}{|c|c|c|c|c|c|c|c|c|}
\hline
r & 3 & 4 & 5 & 6 & 7 & 8 & 9 & 10 \\
\hline
|\alpha| \geq & 103278 & e^{22.3973} & e^{27.2248} & e^{40.5361} & e^{56.4495} & e^{74.9649} & e^{96.0824} & e^{119.8020} \\
\hline
\end{array}
\]

For $r = 5$, we have
\[
V(\alpha) = \log|\alpha| - 24 \log 2 - 20.
\]
We also compute $\mu(\alpha, \varepsilon)$ and $C(\alpha, \varepsilon)$ with $\varepsilon = 0.1$ for the following specific values of $\alpha$:
\[
\begin{array}{|c|c|c|c|c|c|c|c|}
\hline
\alpha & \pm 10^{16} & \pm 10^{17} & \pm 10^{18} & \pm 10^{19} & \pm 10^{20} & \pm 10^{21} & \pm 10^{22} \\
\hline
\mu(\alpha, \varepsilon) & 1740.6055 & 81.2650 & 43.9892 & 31.1889 & 24.7161 & 20.8088 & 18.1940 \\
\hline
C(\alpha, \varepsilon) & e^{-1368.98}/2 & e^{-245.38}/2 & e^{-229.14}/2 & e^{-229.62}/2 & e^{-234.41}/2 & e^{-240.95}/2 & e^{-248.38}/2 \\
\hline
\end{array}
\]

\
 
The present article is organized as follows.
In Subsection $3.1$, 
we review fundamental properties of Pad\'{e} approximation and various elementary facts essential for proving Theorem $\ref{power of log indep}$.
Given that our construction of Pad\'{e} approximation for the product of binomial functions and power of logarithm functions relies on that of exponential functions, Subsection $3.2$ briefly revisits the construction and basic properties of Pad\'{e} approximation for exponential functions.
Using the former subsections, the Pad\'{e} approximations needed in the proof are constructed in Subsection $3.3$.    
Section $4$ is devoted to the analytic estimates pertaining to Pad\'{e} approximation for the  product of binomial functions and power of logarithm functions.
In Section $5$, we then proceed to prove the linear independence criterion (see Lemma $\ref{criterion}$) and deduce our main result.
In Appendix, we delve into the perfectness of Pad\'{e} approximation of Laurent series that is a power series with variable $1/z$ instead of power series. The main purpose of this appendix is to prove the perfectness of Pad\'{e} approximation of polylogarithm function.

\section{Pad\'{e} Approximation} 
In this section, we recall the definition and basic properties of Pad\'{e} approximations of formal power series. Throughout, we fix a field $K$ of characteristic $0$.
We define the order function at $z=0$ by
$${\rm ord}: K((z)) \longrightarrow \mathbb{Z} \cup \{\infty\};\quad \sum_k a_k z^k \mapsto \min\{k \in \mathbb{Z} \cup \{\infty\} \mid a_k \ne 0\}\enspace.$$
Note that for $f \in K((z))$, we have ${\rm ord} \, f = \infty$ if and only if $f = 0$.

\subsection{Pad\'{e} approximations of formal power series}

\begin{lemma} \label{Pade}
Let $m$ be an integer with $m \ge 2$, and let $\boldsymbol{f} = (f_1(z), \ldots, f_m(z)) \in K[[z]]^m$. For $\boldsymbol{n} = (n_1, \ldots, n_m) \in \mathbb{Z}_{\ge 0}^m$, there exists a vector of polynomials $(P_1(z), \ldots, P_m(z)) \in K[z]^m$ satisfying:
\begin{align*}
&(i) \quad (P_1(z), \ldots, P_m(z)) \ne (0, \ldots, 0)\enspace,\\
&(ii) \quad \deg P_j(z) \le n_j \quad \text{for } 1 \le j \le m\enspace,\\
&(iii) \quad {\rm ord} \left( \sum_{j=1}^{m} P_j(z) f_j(z) \right) \ge \sum_{j=1}^m (n_j + 1) - 1\enspace.
\end{align*}
\end{lemma}

\begin{definition}
A tuple of polynomials $(P_1(z), \ldots, P_m(z)) \in K[z]^m$ satisfying conditions $(i)$, $(ii)$, and $(iii)$ is called a Pad\'{e} system of approximants of $\boldsymbol{f}$ of weight $\boldsymbol{n}$.

The power series $\sum_{j=1}^m P_j(z) f_j(z)$ is referred to as a Pad\'{e} approximation of $\boldsymbol{f}$ of weight $\boldsymbol{n}$ (also called the \emph{remainder}).
\end{definition}

\begin{definition} \label{Pade 0}
Let $m \ge 2$ and $\boldsymbol{f} = (f_1(z), \ldots, f_m(z)) \in K[[z]]^m$.

\begin{itemize}
\item[$(i)$] For $\boldsymbol{n} = (n_1, \ldots, n_m) \in \mathbb{Z}_{\ge 0}^m$, we say that $\boldsymbol{n}$ is \emph{normal} with respect to $\boldsymbol{f}$ if every Pad\'{e} approximation $R(z)$ of weight $\boldsymbol{n}$ satisfies
$${\rm ord} \, R(z) = \sum_{j=1}^m (n_j + 1) - 1\enspace.$$

\item[$(ii)$] We say that $\boldsymbol{f}$ is \emph{perfect} if every index $\boldsymbol{n} \in \mathbb{Z}_{\ge 0}^m$ is normal with respect to $\boldsymbol{f}$.
\end{itemize}
\end{definition}

\begin{remark} \label{remark bij} 
Let $\boldsymbol{n} = (n_1, \ldots, n_m) \in \mathbb{Z}_{\ge 0}^m$ and $\boldsymbol{f} = (f_1(z), \ldots, f_m(z)) \in K[[z]]^m$. Let $N = \sum_{j=1}^m (n_j + 1)$ and write $f_j(z) = \sum_{k=0}^\infty f_{j,k} z^k$ for $1 \le j \le m$. 

For a non-negative integer $r$, define the generalized Hankel matrix of $\boldsymbol{f}$ as
\[
H_{r, \boldsymbol{n}}(\boldsymbol{f}) = \begin{pmatrix}
f_{1,0} & 0 & \dots & 0 & \dots & f_{m,0} & 0 & \dots & 0 \\
f_{1,1} & f_{1,0} & \dots & 0 & \dots & f_{m,1} & f_{m,0} & \dots & 0 \\
\vdots & \vdots & \ddots & \vdots & \ddots & \vdots & \vdots & \ddots & \vdots \\
f_{1,r} & f_{1,r-1} & \dots & f_{1,r-n_1} & \dots & f_{m,r} & f_{m,r-1} & \dots & f_{m,r-n_m}
\end{pmatrix} \in {\rm M}_{r+1,N}(K)\enspace,
\]
where $f_{j,k} = 0$ for $k < 0$. Then the following map is a bijection:
\begin{align}
\phi^{(\boldsymbol{n})}_{\boldsymbol{f}}: &\ \ker(H_{N-2,\boldsymbol{n}}(\boldsymbol{f})) \setminus \{\bold{0}\} \rightarrow \left\{(P_j(z)) \in K[z]^m \,\middle|\, \sum_{j=1}^m P_j(z) f_j(z) \text{ is a Pad\'{e} approximation of weight } \boldsymbol{n} \right\} \label{bijection pade} \\
&{}^{t}(p_{1,0}, \ldots, p_{1,n_1}, \ldots, p_{m,0}, \ldots, p_{m,n_m}) \mapsto \left(P_j(z) = \sum_{k=0}^{n_j} p_{j,k} z^k \right)_{1 \le j \le m}\enspace. \nonumber
\end{align}

Note that $\boldsymbol{n}$ is normal with respect to $\boldsymbol{f}$ if and only if $H_{N-1,\boldsymbol{n}}(\boldsymbol{f}) \in {\rm GL}_N(K)$.
\end{remark}

\begin{lemma} \label{cor fund Pade}
Let $m \ge 2$, $\boldsymbol{f} = (f_1(z), \ldots, f_m(z)) \in K[[z]]^m$, and $\boldsymbol{n} = (n_1, \ldots, n_m) \in \mathbb{N}^m$. 
Fix $k$ with $1 \le k \le m$, and define $\tilde{\boldsymbol{n}}_k = (n_1, \ldots, n_{k-1}, n_k + 1, n_{k+1}, \ldots, n_m) \in \mathbb{N}^m$. 
Let $(P_1(z), \ldots, P_m(z)) \in K[z]^m$ be a Pad\'{e} approximant of $\boldsymbol{f}$ of weight $\tilde{\boldsymbol{n}}_k$. Suppose that $\boldsymbol{n}$ is normal with respect to $\boldsymbol{f}$. Then:
$${\rm deg} \, P_k(z) = n_k + 1\enspace.$$
\end{lemma}

\begin{proof}
Let $N = \sum_{j=1}^m (n_j + 1)$. Since $\boldsymbol{n}$ is normal with respect to $\boldsymbol{f}$, we have
$${\rm dim}_K \, \ker(H_{N-1, \boldsymbol{n}}(\boldsymbol{f})) = 0\enspace.$$
Suppose there exists $1 \le k \le m$ and a Pad\'{e} approximant $(P_1(z), \ldots, P_m(z))$ of weight $\tilde{\boldsymbol{n}}_k$ such that ${\rm deg} \, P_k < n_k + 1$.
Let
\[
(\phi^{(\tilde{\boldsymbol{n}}_k)}_{\boldsymbol{f}})^{-1}(P_1(z), \ldots, P_m(z)) = {}^{t}(p_{1,0}, \ldots, p_{1,n_1}, \ldots, p_{k,0}, \ldots, p_{k,n_k}, 0, \ldots, p_{m,0}, \ldots, p_{m,n_m})\enspace,
\]
where $\phi^{(\tilde{\boldsymbol{n}}_k)}_{\boldsymbol{f}}$ is as defined in \eqref{bijection pade}.
Then
\[
{}^{t}(p_{1,0}, \ldots, p_{1,n_1}, \ldots, p_{k,0}, \ldots, p_{k,n_k}, \ldots, p_{m,0}, \ldots, p_{m,n_m}) \in \ker(H_{N-1, \boldsymbol{n}}(\boldsymbol{f})) \setminus \{\bold{0}\}\enspace,
\]
contradicting the invertibility of $H_{N-1, \boldsymbol{n}}(\boldsymbol{f})$. This completes the proof.
\end{proof}

\subsection{Pad\'{e} approximations of exponential functions}
In this section, we recall some properties of Pad\'{e} approximations of exponential functions. 
\begin{proposition} \label{perfect e} $(${\it confer}  \rm{\cite[Theorem $1.2.1$]{J}}$)$
Let $n$ be a non-negative integer and $\omega_0,\ldots,\omega_n$ pairwise distinct elements of $K$. 
Then the formal power series $(e^{\omega_0z},\ldots,e^{\omega_n z})$ is perfect. 
\end{proposition}
Let $\boldsymbol{\omega}=(\omega_0,\ldots,\omega_n)\in K^{n+1}$, where entries are pairwise distinct. 
The explicit Pad\'{e} approximations of $(e^{\omega_0 z},\ldots,e^{\omega_n z})$ are given by K.~Mahler as follows.

\begin{proposition} \label{Pade e} $(${\it confer}  \rm{\cite[p. $242$]{J}}$)$
Let $\boldsymbol{r}=(r_0,\ldots,r_n)\in \Z^{n+1}_{\ge0}$ and $\{p_{h,j}(\boldsymbol{r}, \boldsymbol{\omega})\}_{0\le h \le n, 1\le j \le r_h+1}$ be the subset of $K$ satisfying the following equality:
$$
\dfrac{1}{\prod_{i=0}^n(x-\omega_{i})^{r_{i}+1}}=\sum_{h=0}^n \sum_{j=1}^{r_h+1}\dfrac{p_{h,j}(\boldsymbol{r},\boldsymbol{\omega})}{(x-\omega_h)^j}\enspace.
$$
Then the formal power series $$\mathfrak{R}(z)=\sum_{h=0}^n\left(\sum_{j=0}^{r_h}p_{h,j+1}(\boldsymbol{r},\boldsymbol{\omega})\dfrac{z^j}{j!}\right)e^{\omega_hz}\in K[[z]]\enspace,$$
is a weight $\boldsymbol{r}$ Pad\'{e} approximation of $(e^{\omega_0z},\ldots,e^{\omega_nz})$.
\end{proposition}

\subsection{Pad\'{e} approximations of product of binomial functions and power of logarithm function} 
In this section, we formulate Pad\'{e} approximations for the product of binomial functions and the power of the logarithm function. 
M.~Huttner explored this type of approximation within the context of the Riemann monodromy problem ({\it see} \cite[$3.2)$, p.188]{Huttner}).
Here, we obtain these approximants by adapting Pad\'{e} approximations of exponential functions through the $K$-homomorphism:
$$\mathcal{T}: K[[z]]\longrightarrow K[[z]]; \ f(z)\mapsto f(\log(1+z))\enspace.$$

Note that $\mathcal{T}$ is a $K$-isomorphism with $\mathcal{T}^{-1}:K[[z]]\longrightarrow K[[z]]; \ g(z)\mapsto g(e^z-1)$ and $\mathcal{T}$ is an order preserving map, i.e. 
\begin{align} \label{preserve}
{\rm{ord}}\, f(z)={\rm{ord}}\, \mathcal{T}(f(z)) \ \ \text{for} \ \ f(z)\in K[[z]]\enspace.
\end{align}
We now fix $\omega_1,\ldots,\omega_m\in K$ with $\omega_i-\omega_j\notin \Z$ for $1\le i<j\le m$, along with positive integers $r_1,\ldots,r_m$.
Let us consider the following indices.

\
 
Let $1\le l_i \le r_i$, $(s_{i,1},\ldots,s_{i,l_i})\in \N^{l_i}$ with $s_{i,1}+\cdots+s_{i,l_i}=r_i$ for $1\le i \le m$. 
For non-negative integers $n_{i,1},\ldots,n_{i,l_i}$ with $n_{i,1}>\cdots>n_{i,l_i}$,
\begin{align*}
&\boldsymbol{n}(i)=(\underbrace{n_{i,1},\ldots,n_{i,1}}_{s_{i,1}},\ldots, \underbrace{n_{i,l_i},\ldots,n_{i,l_i}}_{s_{i,l_i}})\in \N^{r_i}\enspace,\\
&\boldsymbol{r}(i)=(\underbrace{r_i-1,\ldots,r_i-1}_{n_{i,l_i}+1},\underbrace{r_i-s_{i,l_i}-1,\ldots,r_i-s_{i,l_i}-1}_{n_{i,l_i-1}-n_{i,l_i}},\ldots,\underbrace{s_{i,1}-1,\ldots,s_{i,1}-1}_{n_{i,1}-n_{i,2}})\in \Z^{n_{i,1}+1}\enspace,
\end{align*}
and $\boldsymbol{n}=(\boldsymbol{n}(1),\ldots,\boldsymbol{n}(m))\in \N^{r_1+\cdots+r_m},\ \ \boldsymbol{r}=(\boldsymbol{r}(1),\ldots,\boldsymbol{r}(m))\in\Z^{n_{1,1}+\cdots+n_{m,1}+m}\enspace.$

\

Denote the set of weight $\boldsymbol{r}$ Pad\'{e} approximation of $(e^{(\omega_i+k_i)z})_{{1\le i \le m, 0\le k_i \le n_{i,1}}}$ by $V_{\boldsymbol{r}}$ and the set of weight $\boldsymbol{n}$ Pad\'{e} approximation of $((1+z)^{\omega_i}{\rm{log}}^{j_i}(1+z))_{{1\le i \le m, 0\le j_i \le r_i-1}}$ by $W_{\boldsymbol{n}}$.

\

We are now prepared to study weight-$\boldsymbol{n}$ Pad\'{e} approximants for the product of binomial functions and powers of the logarithmic function. The key ingredient is the following proposition.

\begin{proposition} \label{diagonal normality log}
We retain the notation as previously established. The $K$-isomorphism $\mathcal{T}$ defines a bijection:
\[\mathcal{T}: V_{\boldsymbol{r}} \longrightarrow W_{\boldsymbol{n}}.\]
In particular, the index $\boldsymbol{n}$ is normal with respect to the system \[((1+z)^{\omega_i}\log^{j_i}(1+z))_{1\le i \le m,\, 0\le j_i \le r_i-1}.\]
\end{proposition}

\begin{proof}
Let $\mathfrak{R}(z) = \sum_{i=1}^m \sum_{k_i=0}^{n_{i,1}} \mathfrak{P}_{i,k_i}(z) e^{(\omega_i+k_i)z} \in V_{\boldsymbol{r}}$.
By Proposition~\ref{perfect e}, we have
\[\mathrm{ord}\, \mathfrak{R}(z) = \sum_{i=1}^m \sum_{v_i=1}^{l_i} (n_{i,v_i}+1) s_{i,v_i}.\]
Applying equation~(\ref{preserve}), we obtain
\[\mathrm{ord}\, \mathcal{T}(\mathfrak{R}(z)) = \sum_{i=1}^m \sum_{v_i=1}^{l_i} (n_{i,v_i}+1) s_{i,v_i}.\]
We claim that there exists a vector of polynomials $(P_{i,j_i}(z))_{1\le i \le m,\, 0 \le j_i \le r_i-1} \in K[z]^{r_1+\cdots+r_m}$ satisfying:
\begin{align*}
\mathcal{T}(\mathfrak{R}(z)) &= \sum_{i=1}^m \sum_{j_i=0}^{r_i-1} P_{i,j_i}(z) (1+z)^{\omega_i} \log^{j_i}(1+z),\\
\deg P_{i,j_i}(z) &\le n_{i,j_i} \quad \text{for } 1 \le i \le m,\, 0 \le j_i \le r_i - 1.
\end{align*}
This implies $\mathcal{T}(\mathfrak{R}(z)) \in W_{\boldsymbol{n}}$. Since $\mathcal{T}$ is linear, we may assume $i = 1$ and define $\omega_1 = \omega$, $r_1 = r$, $l_1 = l$, $s_{1,j} = s_j$, $n_{1,j} = n_j$ (for $1 \le j \le l$), and $\mathfrak{P}_{1,k}(z) = \mathfrak{P}_k(z) = \sum_{j=0}^{r-1} p_{j,k} z^j$.
Then we have:
{\small{\begin{align*}
\sum_{k=0}^{n_{1}}\mathfrak{P}_{k}(z)e^{(\omega+k)z}
&=\sum_{k=0}^{n_{l}}\left(\sum_{j=0}^{r-1}p_{j,k}z^{j}\right)e^{(\omega+k)z}+\sum_{k=n_{l}+1}^{n_{l-1}}\left(\sum_{j=0}^{r-s_{l}-1}p_{j,k}z^{j}\right)e^{(\omega+k)z}\\
&+\cdots +\sum_{k=n_{3}+1}^{n_{2}}\left(\sum_{j=0}^{s_{1}+s_{2}-1}p_{j,k}z^{j}\right)e^{(\omega+k)z}+\sum_{k=n_{2+1}}^{n_{1}}\left(\sum_{j=0}^{s_{1}-1}p_{j,k}z^{j}\right)e^{(\omega+k)z}\\
&=\sum_{k=0}^{n_{1}}\left(\sum_{j=0}^{s_{1}-1}p_{j,k}z^{j}\right)e^{(\omega+k)z}+\sum_{k=0}^{n_{2}}\left(\sum_{j=s_{1}}^{s_{2}+s_{1}-1}p_{j,k}z^{j}\right)e^{(\omega+k)z}\\
&+\cdots +\sum_{k=0}^{n_{l-1}}\left(\sum_{j=r-s_{l}-s_{l-1}}^{r-s_{l}-1}p_{j,k}z^{j}\right)e^{(\omega+k)z}+\sum_{k=0}^{n_{l}}\left(\sum_{j=r-s_{l}}^{r-1}p_{j,k}z^{j}\right)e^{(\omega+k)z}\enspace.
\end{align*}}}
Acting $\mathcal{T}$ on the both sides of the above equalities, 
{\small{\begin{align*}
&\mathcal{T}\left(\sum_{k=0}^{n_{1}}\mathfrak{P}_{k}(z)e^{(\omega+k)z}\right)\\
&=\sum_{k=0}^{n_{1}}\left(\sum_{j=0}^{s_{1}-1}p_{j,k}\log^{j}(1+z)\right)(1+z)^{\omega+k}+\sum_{k=0}^{n_{2}}\left(\sum_{j=s_{1}}^{s_{2}+s_{1}-1}p_{j,k}\log^{j}(1+z)\right)(1+z)^{\omega+k}+\cdots \\
&+\sum_{k=0}^{n_{l-1}}\left(\sum_{j=r-s_{l}-s_{l-1}}^{r-s_{l}-1}p_{j,k}\log^j(1+z)\right)(1+z)^{\omega+k}+\sum_{k=0}^{n_{l}}\left(\sum_{j=r-s_{l}}^{r-1}p_{j,k}\log^{j}(1+z)\right)(1+z)^{\omega+k}\\
&=\sum_{j=0}^{s_{1}-1}\left(\sum_{k=0}^{n_{1}}{p_{j,k}(1+z)^k}\right)(1+z)^{\omega}\log^{j}(1+z)+\sum_{j=s_{1}}^{s_{2}+s_{1}-1}\left(\sum_{k=0}^{n_{2}}{p_{j,k}(1+z)^k}\right)(1+z)^{\omega}\log^{j}(1+z)\\
&+\cdots+\sum_{j=r-s_{l}}^{r-1}\left(\sum_{k=0}^{n_{l}}{p_{j,k}(1+z)^k}\right)(1+z)^{\omega}\log^{j}(1+z)\enspace.
\end{align*}}}
By tracing the above equalities in reverse, we also conclude that 
\[
\mathcal{T}^{-1}(R(z))\in V_{\boldsymbol{r}} \ \text{for} \ R(z)\in W_{\boldsymbol{n}}.
\] 
This completes the proof of Proposition~\ref{diagonal normality log}. 
\end{proof}

Let $k$ be an integer with $1\le k \le r_1+\cdots +r_m$. 
Applying Proposition $\ref{diagonal normality log}$ for $$\boldsymbol{n}_k=(\underbrace{n+1,\ldots,n+1}_k,n,\ldots,n)\in \N^{r_1+\cdots+r_m}\enspace,$$ we obtain the following explicit weight $\boldsymbol{n}_k$ Pad\'{e} approximants of $((1+z)^{\omega_i}\log^{j_i}(1+z))_{\substack{1\le i \le m \\ 0\le j_i \le r_i-1}}$.

\begin{corollary}\label{main pade}
We use the same notation as above. Put $k=r_1+\cdots+r_u+s_{u+1}$ for some $0\le u \le m-1$ and $1\le s_{u+1} \le r_{u+1}$, where we set $r_1+\cdots+r_u=0$ if $u=0$.
Define a rational function 
\begin{align}\label{fk}
F_k(x)=\dfrac{1}{\prod_{i=1}^u\prod_{h=0}^{n+1}(x-\omega_i-h)^{r_i}}\cdot \dfrac{1}{(x-\omega_{u+1}-n-1)^{s_{u+1}}} \cdot \dfrac{1}{\prod_{i'={u+1}}^m \prod_{h'=0}^{n}(x-\omega_{i'}-h')^{r_{i'}}}\enspace.
\end{align}
Let $\{p_{k,h,i,j_i}(\boldsymbol{\omega})\}\subset K$ such that
\begin{align} \label{def akij}
F_k(x)=\sum_{i=1}^m\sum_{h=0}^{n+1}\sum_{j_i=1}^{r_i}\dfrac{p_{k,h,i,j_i}(\boldsymbol{\omega})}{(x-\omega_i-h)^{j_i}}\enspace.
\end{align}
Define polynomials 
\begin{align*}
P_{k,i,j_i}(z)=\sum_{h=0}^{n+1}\dfrac{p_{k,h,i,j_i+1}(\boldsymbol{\omega})}{j_i!}(1+z)^h \ \ \text{for} \ \ 1\le i \le m, \ 0\le j_i \le r_i-1\enspace.
\end{align*}
Then the vector $(P_{k,i,j_i}(z))_{\substack{1\le i \le m \\ 0\le j_i \le r_i-1}}$ is a weight $\boldsymbol{n}_k$ Pad\'{e} approximant of $((1+z)^{\omega_i}\log^{j_i}(1+z))_{{1\le i \le m, 0\le j_i \le r_i-1}}$.
\end{corollary}
\begin{proof}
Define
\begin{align*}
&\tilde{\boldsymbol{n}}(i)=\begin{cases}
(n+1,\ldots,n+1)\in \N^{r_i} & \ \text{if} \ 1\le i \le u\\
(\underbrace{n+1,\ldots, n+1}_{s_{u+1}},n,\ldots,n) \in \N^{r_{u+1}} & \ \text{if} \ i=u+1\\
(n,\ldots,n) \in \N^{r_i} & \ \text{if} \ u+2 \le i \le m\enspace,
\end{cases}\\
&\tilde{\boldsymbol{r}}(i)=\begin{cases}
(r_i-1,\ldots,r_i-1)\in \Z^{n+2} & \ \text{if} \ 1\le i \le u\\
(r_{u+1}-1,\ldots,r_{u+1}-1,s_{u+1}-1) \in \Z^{n+2} & \ \text{if} \ i=u+1\\
(r_i-1,\ldots,r_i-1) \in \Z^{n+1} & \ \text{if} \ u+2 \le i \le m\enspace.
\end{cases}
\end{align*}
Then the indice $\boldsymbol{n}_k$ is $(\tilde{\boldsymbol{n}}(1),\ldots,\tilde{\boldsymbol{n}}(m))$.
Put $$\mathfrak{P}_{k,i,h}(z)=\sum_{j_i=0}^{r_i-1}\dfrac{p_{k,h,i,j_i}(\boldsymbol{\omega})}{j_i!}z^{j_i} \ \ \text{for} \ \ 1\le i \le m, \ 0\le h \le n+1\enspace.$$
By Proposition $\ref{Pade e}$, the vector of polynomials $(\mathfrak{P}_{k,i,h}(z))$ is a weight $\boldsymbol{r}=(\tilde{\boldsymbol{r}}(1),\ldots, \tilde{\boldsymbol{r}}(m))$ Pad\'{e} approximant of 
$(e^{(\omega_i+h)z})_{\substack{1\le i \le u+1 \\ 0\le h \le n+1}} \cup (e^{(\omega_{i'}+h')z})_{\substack{u+2 \le i' \le m \\ 0 \le h' \le n}}.$ 
Put $$\mathfrak{R}_k(z)=\sum_{i=1}^{u+1}\sum_{h=0}^{n+1}\mathfrak{P}_{k,i,h}(z)e^{(\omega_i+h)z}+\sum_{i'=u+2}^{m}\sum_{h'=0}^{n}\mathfrak{P}_{k,i',h'}(z)e^{(\omega_{i'}+h')z}\enspace.$$
Subsequently, we find that $\mathfrak{R}_k(z)\in V_{\boldsymbol{r}}$ and, based on Proposition $\ref{diagonal normality log}$, conclude 
$$\mathcal{T}(\mathfrak{R}_k(z))=\sum_{i=1}^m \sum_{j_i=0}^{r_i-1} P_{k,i,j_i}(z)(1+z)^{\omega_i}\log^{j_i}(1+z)\in W_{\boldsymbol{n}_k}\enspace.$$
This completes the proof of Corollary $\ref{main pade}$.
\end{proof}

\section{Estimates}
From this section to the last, we use the following notations. 
Let $m,r_1,\ldots,r_m\in \N$. Let $K$ be an algebraic number field and $\omega_1,\ldots,\omega_m\in \Q$ with $0\le \omega_1<\cdots<\omega_m<1$.
We put $r=\max(r_1,\ldots,r_m)$.
Unless stated otherwise, the Landau symbol $o$ refers to the limit as $n$ tends to infinity.

\medskip

Let $k,n$ be non-negative integers with $1\le k \le {{\sum_{i=1}^m}}r_i$ and 
$$P_{k,i,j_i,n+1}(z)=P_{k,i,j_i}(z)=\sum_{h=0}^{n+1}\dfrac{p_{k,h,i,j_i+1}(\boldsymbol{\omega})}{j_i!}(1+z)^h\in \Q[z] \ \ \text{for} \ \ 1\le i \le m, \ 0\le j_i \le r_i-1\enspace,$$ the polynomials defined in Corollary $\ref{main pade}$.
Put 
\begin{align*}
R_{k,n+1}(z)=R_{k}(z)=\sum_{i=1}^m\sum_{j_i=0}^{r_i-1}P_{k,i,j_i}(z)(1+z)^{\omega_i}\log^{j_i}(1+z)\in \Q[[z]]\enspace.
\end{align*}
Firstly we give a positive integer which annihilates the denominator of each coefficients of $P_{k,i,j_i}(z)$. 
\begin{lemma} \label{trivial}
Let $g_1(t),\ldots,g_m(t)\in \Q[[t]]$. Assume there exist $A_i,B_i\in \Z$ with $A_ig_i(B_it)\in \Z[[t]]$ for $1\le i \le m$.
Put $g_1(t)\cdots g_m(t)=\sum_{k=0}^{\infty}c_kt^k$. Then $$\prod_{i=1}^mA_iB^k_i\cdot c_k\in \Z \ \ \text{for} \  \ k\ge 0 \enspace.$$
\end{lemma}
\begin{proof}
Put $g_i(t)=\sum_{k=0}^{\infty}c_{i,k}t^k$. Notice that we have $A_iB^k_ic_{i,k}\in \Z$ by the assumption.
Since $$c_k=\sum_{\substack{1\le j_1,\ldots,j_m\le k \\ j_1+\cdots+j_m=k}}c_{1,j_1}\ldots c_{m,j_m}\enspace,$$
we obtain $$\prod_{i=1}^mA_iB^k_i\cdot c_k\in \Z\enspace.$$
\end{proof}

For a positive integer $n$ and $\omega\in \Q\setminus \Z_{<0}$, put 
$$D_{n}(\omega)={\rm{den}}\left(\dfrac{0!}{(\omega)_0},\ldots,\dfrac{n!}{(\omega)_n}\right), \ \ d_n(\omega)={\rm{den}}\left(\dfrac{1}{\omega},\ldots,\dfrac{1}{\omega+n-1}\right)\enspace.$$
For $\omega=1$, the integer $d_n(1)$ is the least common multiple of $1,\ldots,n$ and denote $d_n$.
\begin{lemma} \label{denominator}
Let $i,n$ be non-negative integers with $1\le i \le m$.  
We have
$$(r_i-1)!(n+2)!^{r_1+\cdots+r_m}\left[\prod_{\substack{1\le i'\le m \\ i'\neq i}}D_{n+2}(\omega_i-\omega_{i'})^2d_{n+2}(\omega_i-\omega_{i'})\right]^{r_i'}d_{n+1}^{r_i}P_{k,i,j_i}(z)\enspace,$$
for $1\le k \le  r_1+\cdots+r_m$ and $0\le j_i\le r_i-1$.

\

In particular we have $$(r-1)!(n+2)!^{r_1+\cdots+r_m}\left(\left[\prod_{1\le i_1<i_2\le m} D_{n+2}(\omega_{i_2}-\omega_{i_1})^2d_{n+2}(\omega_{i_2}-\omega_{i_1})\right]d_{n+1}\right)^{r}P_{k,i,j_i}(z)\in \Z[z]\enspace,$$
for $1\le k \le r_1+\cdots+r_m$, $1\le i \le m$ and $0\le j_i\le r_i-1$.
\end{lemma}
\begin{proof}
We only consider the case of $k=r_1+\cdots+r_m$, since we obtain the assertion by the similar way for general $k$.
For $1\le i'\le m$ and $0\le h \le n+1$, we define rational functions 
{\small{
$$
G_{i',h}(t)=
\begin{cases}
&\left[\dfrac{(-1)^{(n+1-h)}}{h!(n+1-h)!}\cdot \dfrac{1}{t}\cdot \displaystyle{\prod_{\lambda=1}^h}\left(1+\dfrac{t}{\lambda}\right)^{-1}\cdot \displaystyle{\prod_{\lambda=1}^{n+1-h}}\left(1-\dfrac{t}{\lambda}\right)^{-1}\right]^{r_i}        \ \ \ \text{if} \ i'=i\\
& \\
&\left[\dfrac{1}{(\omega_i-\omega_{i'})_{h+1}(\omega_{i'}-\omega_i+1)_{n+1-h}}\cdot \displaystyle{\prod_{\lambda=0}^h}\left(1+\dfrac{t}{\omega_{i'}-\omega_i+\lambda}\right)^{-1}\cdot 
\displaystyle{\prod_{\lambda=1}^{n+1-h}}\left(1-\dfrac{t}{\omega_{i'}-\omega_i-\lambda}\right)^{-1}\right]^{r_{i'}}\enspace,      
\end{cases}
$$}}
where second one is for $i'\neq i$.
Then $$F_{r_1+\cdots+r_m}(x)=\prod_{i'=1}^mG_{i',h}(x-\omega_{i'}-h)\enspace,$$
(see the definition of $F_{r_1+\cdots+r_m}(x)$ in $(\ref{fk})$) and
\begin{align} \label{denomi 1}
\begin{cases}
(n+1)!^{r_{i}}t^{r_i}G_{i,h}(d_{n+1}t)\in \Z[[t]] & \ \ \text{if} \ i'=i\enspace,\\
(n+2)!^{r_{i'}}D_{n+1}(\omega_i-\omega_{i'})^{2r_{i'}}G_{i',h}\left(d_{n+2}(\omega_i-\omega_{i'})t \right)\in \Z[[t]] & \ \ \text{if} \ i'\neq i\enspace,
\end{cases}
\end{align}
for $0\le h \le n+1$.
By the definition of $p_{n+1,h,i,j_i}(\boldsymbol{\omega})$ (see $(\ref{def akij})$), using Equation $(\ref{denomi 1})$ and Lemma $\ref{trivial}$, $$(n+2)!^{r_1+\cdots+r_m}\left[\prod_{\substack{1\le i \le m \\ i'\neq i}}D_{n+2}(\omega_i-\omega_{i'})^2d_{n+2}(\omega_i-\omega_{i'})\right]^{r_i'}d_{n+1}^{r_i}p_{n+1,h,i,j_i}(\boldsymbol{\omega})\in \Z\enspace.$$
Therefore, by the definition of $P_{k,i,j_i}(z)$, we obtain the assertion.
\end{proof}
Our next aim is to find an estimate for the coefficients of $P_{k,i,j_i}(z)$.
\begin{lemma} \label{absolute value a}
We have
\begin{align*}
\max_{\substack{1\le k \le r_1+\cdots+r_m \\ 1 \le i \le m, 0\le j_i \le r_i-1 \\ 0\le h \le n+1}}|p_{k,h,i,j_i+1}(\boldsymbol{\omega})|\le \left(\dfrac{2^n}{(n+1)!}\right)^{r_1+\cdots+r_m}\cdot e^{o(n)} \enspace.
\end{align*}
\end{lemma}
\begin{proof}
Fix $1\le k \le r_1+\cdots+r_m, 1 \le i \le m, 0\le j_i \le r_i-1$ and $0\le h \le n+1$. Put $$R=\begin{cases} {\displaystyle{\min_{1\le i \le m-1}}}\{\omega_{i+1}-\omega_i\} & \ \ \text{if} \ m\ge 2\enspace,\\ 1 & \ \ \text{if} \ m=1\enspace.\end{cases}$$ 
By the residue theorem,  $p_{k,h,i,j_i+1}(\boldsymbol{\omega})$ may also be written as 
\begin{align}\label{int rep a}
p_{k,h,i,j_i+1}(\boldsymbol{\omega})=\dfrac{1}{2\pi\sqrt{-1}} \int_{\mathcal{C}_{i,h}}(\zeta-\omega_i-h)^{j_i}F_k(\zeta) \, d\zeta \enspace,
\end{align}
where $F_k$ is the rational function defined in Equation $(\ref{fk})$ and $\mathcal{C}_{i,h}$ may be chosen as the circle of center $\zeta=\omega_i+h$ and radius $R/2$, described in the positive direction.
We now estimate $\displaystyle{\max_{x\in \C, |x-\omega_i-h|=R/2}} \left|F_k(x)\right|$.

\

In what follows, we only consider the case of $k=r_1+\cdots+r_m$, as similar arguments apply for any $1\le k\le r_1+\cdots+r_m-1$.
Since
\begin{align*}
F_{r_1+\cdots+r_m}(x)=\dfrac{1}{\prod_{i'=1}^m\prod_{h'=0}^{n+1}(x-\omega_{i'}-h')^{r_{i'}}}\enspace,
\end{align*}
it is enough to estimate 
$$\max_{\substack{x\in \C \\ |x-\omega_i-h|=R/2}} \left|\dfrac{1}{\prod_{h'=0}^{n+1}(x-\omega_{i'}-h')^{r_{i'}}}\right| \ \ \text{for} \ \ 1\le i' \le m \enspace.$$

\

\textbf{Case $\rm{\,I\,}$: $i'=i$.}

Let $0\le h'\le n+1$ and $x\in\{z\in\C\mid |z-\omega_i-h|=R/2\}$. Since,
\begin{align*}
|x-\omega_i-h'|&=|x-\omega_i-h+h-h'|\ge ||h-h'|-|x-\omega_i-h||
\ge
\begin{cases}
|h-h'|-1/2 & \ \ \text{if} \ h'\neq h\\
R/2 & \ \ \text{if} \ h'=h\enspace,
\end{cases}
\end{align*}
we obtain
$$
\max_{\substack{x\in \C \\ |x-\omega_i-h|=R}} \left|\dfrac{1}{\prod_{h'=0}^{n+1}(x-\omega_{i}-h')^{r_{i'}}}\right|\le \left(\dfrac{2}{R}\right)^{r_i}\cdot \left(\dfrac{4\cdot 4^nh!(n+1-h)!}{(2h)!(2(n+1-h)!)}\right)^{r_i}\enspace.
$$
Using the equality
\begin{align*}
\dfrac{h!(n+1-h)!}{(2h)!(2(n+1-h))!}&=\frac{(2(n+1))!}{(2h)!(2(n+1-h))!}\cdot\dfrac{(n+1)!(n+1)!}{(2(n+1))!}\cdot\dfrac{h!(n+1-h)!}{(n+1)!}\cdot \dfrac{1}{(n+1)!}\\
&=\binom{2(n+1)}{2h}\cdot {\binom{2(n+1)}{n+1}}^{-1}\cdot {\binom{n+1}{h}}^{-1}\cdot \dfrac{1}{(n+1)!}\enspace,
\end{align*}
and the inequality 
$$\binom{2(n+1)}{2h}\cdot {\binom{2(n+1)}{n+1}}^{-1}\cdot {\binom{n+1}{h}}^{-1}\le\dfrac{n+2}{2^{n+1}} \ \ \text{for} \ \ 0\le h \le n+1\enspace,$$
(see Proof of {\cite[Theorem $1$, p.~$376$]{M2}}), we conclude
\begin{align} \label{i'=i}
\max_{\substack{x\in \C \\ |x-\omega_i-h|=R}} \left|\dfrac{1}{\prod_{h'=0}^{n+1}(x-\omega_{i}-h')^{r_{i'}}}\right|\le \left(\dfrac{4(n+2)}{R}\right)^{r_i} \cdot \left(\dfrac{2^n}{(n+1)!}\right)^{r_i}\enspace.
\end{align}

\

\textbf{Case $\rm{I\hspace{-.01em}I}$: $i'>i$.}

Let $0\le h'\le n+1$ and $x\in\{z\in\C\mid |z-\omega_{i'}-h|=R/2\}$. Since,
\begin{align*}
|x-\omega_{i'}-h'|=|x-\omega_i-h+w_{i}-w_{i'}+h-h'|&\ge ||h-h'+\omega_{i}-\omega_{i'}|-|x-\omega_i-h||\\
&\ge
\begin{cases}
h'-h+\omega_{i'}-\omega_i-R/2 & \ \ \text{if} \ h+1\le h'\\
|h-h'+\omega_i-\omega_{i'}-R/2| & \ \ \text{if} \ h'=h-1,h\\
h-h'+\omega_i-\omega_{i'}-R/2 & \ \ \text{if} \ 0\le h'\le h-2
\end{cases}\\
&\ge 
\begin{cases}
h'-h-1/2 & \ \ \text{if} \ h+1\le h'\\
C_1 & \ \ \text{if} \ h'=h-1,h\\
h-h'-3/2 & \ \ \text{if} \ 0\le h'\le h-2\enspace,
\end{cases}
\end{align*}
where $C_1=\min\{|\omega_i-\omega_{i'}-R/2|,|1+\omega_i-\omega_{i'}-R/2|\}$.
Thus we obtain
$$
\max_{\substack{x\in \C \\ |x-\omega_i-h|=R/2}} \left|\dfrac{1}{\prod_{h'=0}^{n+1}(x-\omega_{i'}-h')^{r_{i'}}}\right|\le 
\dfrac{1}{C^{r_{i'}}_1}\cdot \left(\dfrac{4^n(n+1-h)!(h-1)!)}{(2(n+1-h))!(2h-2)!}\right)^{r_{i'}}\enspace.
$$ 
Using the equality
\begin{align*}
\dfrac{(n+1-h)!(h-1)!)}{(2(n+1-h))!(2h-2)!}=\binom{2n}{2(h-1)}\cdot {\binom{2n}{n}}^{-1}{\binom{n}{h-1}}^{-1}\cdot \dfrac{1}{n!}
\end{align*}
and the inequality
$$
\binom{2n}{2(h-1)}\cdot {\binom{2n}{n}}^{-1}{\binom{n}{h-1}}^{-1}\le \dfrac{n+1}{2^n}\enspace,
$$
we conclude
\begin{align} \label{i'>i}
\max_{\substack{x\in \C \\ |x-\omega_i-h|=R/2}} \left|\dfrac{1}{\prod_{h'=0}^{n+1}(x-\omega_{i'}-h')^{r_{i'}}}\right|\le  \left(\dfrac{(n+1)^2}{C_1}\right)^{r_{i'}} \cdot \left(\dfrac{2^n}{(n+1)!}\right)^{r_{i'}}\enspace.
\end{align}

\

\textbf{Case $\rm{I\hspace{-.15em}I\hspace{-.15em}I}$: $i'<i$.}

By the similar arguments in the case of $i'>i$, we obtain
$$
\max_{\substack{x\in \C \\ |x-\omega_i-h|=R/2}} \left|\dfrac{1}{\prod_{h'=0}^{n+1}(x-\omega_{i'}-h')^{r_{i'}}}\right|
\le 
\left(\dfrac{1}{C_2}\right)^{r_{i'}}\cdot  \left(\dfrac{4^n(n+1-h)!(h-1)!)}{(2(n+1-h))!(2h-2)!}\right)^{r_{i'}}
$$ where  $C_2=\min\{|\omega_i-\omega_{i'}-R/2|,|1+\omega_{i'}-\omega_{i}-R/2|\}$. 
Therefore,
\begin{align*} 
\max_{\substack{x\in \C \\ |x-\omega_i-h|=R/2}} \left|\dfrac{1}{\prod_{h'=0}^{n+1}(x-\omega_{i'}-h')^{r_{i'}}}\right|\le   \left(\dfrac{(n+1)^2}{C_2}\right)^{r_{i'}} \cdot \left(\dfrac{2^n}{(n+1)!}\right)^{r_{i'}} \enspace.
\end{align*}
Using Equations $(\ref{i'=i})$ and $(\ref{i'>i})$ and applying the above inequality to Equation $(\ref{int rep a})$, we obtain the desire inequality.
This completes the proof of Lemma $\ref{absolute value a}$.
\end{proof}

\begin{proposition} \label{absolute value P}
Let $v$ be a place of $K$ and $\alpha\in K\setminus\{0\}$. Put $\varepsilon_v=1$ if $v\in \mathfrak{M}^{\infty}_K$ and $\varepsilon_v=0$ if $v\in\mathfrak{M}^f_{K}$
and
\begin{align} \label{E_n}
D_{n+1}=(r-1)!(n+2)!^{r_1+\cdots+r_m}\left(\left[\prod_{1\le i_1<i_2\le m} D_{n+2}(\omega_{i_2}-\omega_{i_1})^2d_{n+2}(\omega_{i_2}-\omega_{i_1})\right]d_{n+1}\right)^{r}\enspace.
\end{align}
Then
{\small{\begin{align*}
\max_{\substack{1\le k \le r_1+\cdots+r_m \\ 1 \le i \le m, 0\le j_i \le r_i-1}}\log |D_{n+1}\alpha^{n+1}P_{k,i,j_i}(\alpha^{-1})|_v&\le (n+1){\rm{h}}_v(\alpha)+\varepsilon_vo(n)+
\dfrac{\varepsilon_v(r_1+\cdots+r_m+1)[K_v:\R]n}{[K:\Q]}\log \left(2\right)\\
&+\varepsilon_v r \log\left|\left[\prod_{1\le i_1<i_2\le m} D_{n+2}(\omega_{i_2}-\omega_{i_1})^2d_{n+2}(\omega_{i_2}-\omega_{i_1})\right]d_{n+1}\right|_v \enspace.
\end{align*}}}
\end{proposition}
\begin{proof}
First, we consider the case $v_0\in\mathfrak{M}^{f}_K$. According to Lemma $\ref{denominator}$, $D_{n+1}P_{k,i,j_i}(z)\in \Z[z]$ with a degree at most $n+1$.
Thus, the strong triangle inequality enable us to obtain the assertion.

We consider the case $v_0\in\mathfrak{M}^{\infty}_K$.
By the definition of $P_{k,i,j_i}(z)$,
\begin{align}\label{P rep}
D_{n+1}\alpha^{n+1}P_{k,i,j_i}(\alpha^{-1})=D_{n+1}\sum_{\lambda=0}^{n+1}\left(\sum_{h=\lambda}^{n+1}\dfrac{p_{k,h,i,j_i}(\boldsymbol{\omega})}{j_i!}\binom{h}{\lambda}\right)\alpha^{n+1-\lambda}\enspace.
\end{align}
By Lemma $\ref{absolute value a}$ and the triangle inequality, 
$$\max_{0\le \lambda\le n+1} \left|\sum_{h=\lambda}^{n+1}\dfrac{p_{k,h,i,j_i}(\boldsymbol{\omega})}{j_i!}\binom{h}{\lambda}\right|\le 2^{n+1}\cdot \left(\dfrac{2^n}{(n+1)!}\right)^{r_1+\cdots+r_m}\cdot e^{o(n)} \ \ 
(n\to\infty)\enspace.$$
Combining the above inequality with the triangle inequality again for Equation $(\ref{P rep})$ yields the desire inequality. 
This completes the proof of Proposition $\ref{absolute value P}$.
\end{proof}
We conclude this section by determining an upper bound for $R_k(z)$.
\begin{lemma} \label{reminder term}
Let $v_0$ be a place of $K$ and $\alpha\in K$ with 
\begin{align} \label{cond a}
|\alpha|_{v_0}>
\begin{cases} 
1 \ & \ \ \text{if} \ v_0\in \mathfrak{M}^{\infty}_K \ \text{or} \ v_0\in \mathfrak{M}^{f}_K \ \text{and} \ \  p_{v_0}\nmid {\rm{den}}(\omega_1,\ldots,\omega_m)\enspace,\\ 
|p_{v_0}|^{-p_{v_0}/(p_{v_0}-1)}_{v_0} \ & \ \ \text{otherwise}\enspace.
\end{cases}
\end{align}
Let $D_{n+1}$ be the integer defined in $(\ref{E_n})$.
Then
{\small{\begin{align*}
\max_{\substack{1\le k \le r_1+\cdots+r_m}}\log|D_{n+1}\alpha^{n+1}R_k(\alpha^{-1})|_{v_0}&\le -(n+1)(r_1+\cdots+r_m-1){\rm{h}}_{v_0}(\alpha)+{o(n)}\\
&+\dfrac{\varepsilon_v(r_1+\cdots+r_m+1)[K_v:\R]n}{[K:\Q]}\log \left(2\right)\\
&+\varepsilon_v r \log\left|\left[\prod_{1\le i_1<i_2\le m} D_{n+2}(\omega_{i_2}-\omega_{i_1})^2d_{n+2}(\omega_{i_2}-\omega_{i_1})\right]d_{n+1}\right|_{v_0}\\
&+(\varepsilon_v-1)
\begin{cases} 
0 & \ \text{if} \ p_{v_0}\nmid {\rm{den}}(\omega_i)_{1\le i \le m}\\
\dfrac{(n+1)(r_1+\cdots+r_m)p_{v_0}}{p_{v_0}-1}\log|p_{v_0}|_{v_0} & \ \text{if} \ p_{v_0}\mid {\rm{den}}(\omega_i)_{1\le i \le m}\enspace.
\end{cases}
\end{align*}}}
\end{lemma}
\begin{proof}
We develop $D_{n+1}R_k(z)$ into a power series 
$$D_{n+1}R_k(z)=\sum_{\lambda=k+(n+1)(r_1+\cdots+r_m)-1}^{\infty}r_{k,\lambda}z^{\lambda}\enspace.$$
Notice, by assumption $(\ref{cond a})$ of $\alpha$, $R_k(\alpha^{-1})$ converges in $K_{v_0}$.
First we estimate $|r_{k,\lambda}|_{v_0}$ step by step.

\

Let $j$ be a non-negative integer. Put $\log^j(1+z)=\sum_{\lambda=0}^{\infty}a_{j,\lambda}z^{\lambda}$. Then $a_{j,\lambda}=0$ if $\lambda< j$ and 
$$a_{j,\lambda}=(-1)^{\lambda j}\sum_{\substack{1\le \lambda_1,\ldots.\lambda_j\le \lambda\\ \lambda_1+\ldots+\lambda_j=\lambda}}\dfrac{1}{\lambda_1\cdots \lambda_j}\in \Q \ \ \text{for} \ \lambda\ge j \enspace.$$
The following inequalities 
\begin{align*}
&{\mathrm{Card}}\{(\lambda_1,\ldots,\lambda_j)\in \N^j\mid \lambda_1+\cdots+\lambda_j=\lambda\}=\binom{\lambda-1}{j-1}  \ \ \text{for} \ \lambda\ge j\enspace,\\
&\left|\dfrac{1}{\lambda}\right|_{v_0}\le \lambda^{\tfrac{[K_{v_0}:\Q_{p_{v_0}}]}{[K:\Q]}}  \ \text{if} \ v_0\in \mathfrak{M}^{f}_K\enspace,
\end{align*}
allow us to obtain
\begin{align} \label{ajl}
|a_{j,\lambda}|_{v_0}\le 
\begin{cases}
{\binom{\lambda-1}{j-1}}^{\tfrac{[K_{v_0}:\R]}{[K:\Q]}} & \ \text{for} \ v_0\in \mathfrak{M}^{\infty}_K\enspace,\\
\lambda^{\tfrac{j[K_{v_0}:\Q_{p_{v_0}}]}{[K:\Q]}} & \ \text{for} \ v_0\in \mathfrak{M}^{f}_K\enspace,
\end{cases}
\end{align}
for $\lambda\ge 1$.
Let $\omega\in \Q\cap [0,1)$. Put $(1+z)^{\omega}\log^j(1+z)=\sum_{\lambda=j}^{\infty}b_{j,\lambda}(\omega)z^{\lambda}$. Then
\begin{align} \label{bjl}
b_{j,\lambda}(\omega)=\sum_{\substack{0\le \lambda_1,\lambda_2\le \lambda\\ \lambda_1+\lambda_2=\lambda}}\binom{\omega}{\lambda_1}a_{j,\lambda_2}\enspace.
\end{align}
We estimate $|b_{j,\lambda}(\omega)|_{v_0}$. Since $\omega\in [0,1)$, we have $$\left|\binom{\omega}{\lambda}\right|\le 1 \ \ \text{for} \ \ \lambda\ge 0\enspace.$$
Put $\mu_{\lambda}(\omega)=\prod_{q \, \text{prime}, q|{\rm{den}}(\omega)}q^{\lambda+\lfloor \lambda/(q-1)\rfloor}$. 
By \cite[Lemma $2.2$]{Beu}, 
$$\mu_{\lambda}(\omega)\binom{\omega}{\lambda_1}\in \Z \ \ \text{for} \ \ 0\le \lambda_1\le \lambda\enspace.$$
Using Equation $(\ref{ajl})$ and the above relations for Equation $(\ref{bjl})$, 
\begin{align} \label{bjl2}
|b_{j,\lambda}(\omega)|_{v_0}\le \begin{cases}
(\lambda-j+1){\binom{\lambda-1}{j-1}}^{\tfrac{[K_{v_0}:\R]}{[K:\Q]}} & \ \ \text{if} \ \ v_0\in \mathfrak{M}^{\infty}_K\enspace,\\
|\mu_{\lambda}(\omega)|^{-1}_{v_0}\lambda^{\tfrac{j[K_{v_0}:\Q_{p_{v_0}}]}{[K:\Q]}}  & \ \ \text{if} \ \ v_0\in \mathfrak{M}^{f}_K\enspace.
\end{cases}
\end{align}
Put $D_{n+1}P_{k,i,j_i}(z)(1+z)^{\omega_i}\log^{j_i}(1+z)=\sum_{\lambda=0}^{\infty}c_{k,i,j_i,\lambda}z^{\lambda}$. We estimate $|c_{k,i,j_i,\lambda}|_{v_0}$.

\smallskip

In the case of $v_0\in \mathfrak{M}^{\infty}_K$, Equation $(\ref{bjl2})$, Lemma $\ref{absolute value a}$ and triangle inequality yield
\begin{align} \label{cjl inf}
|c_{k,i,j_i,\lambda}|_{v_0}\le e^{o(n)}\cdot \left|2^{n+1}\right|_{v_0} \cdot \left|\dfrac{2^nD_{n+1}}{(n+1)!}\right|^{r_1+\cdots+r_m}_{v_0}\cdot \lambda^{(r-1)\tfrac{[K_{v_0}:\R]}{[K:\Q]}+1} \enspace.
\end{align}

\smallskip

In the case $v_0\in \mathfrak{M}^{f}_K$, since we have $D_{n+1}P_{k,i,j_i}(z)\in \Z[z]$ by Lemma $\ref{denominator}$, combining Equation $(\ref{bjl2})$ and the strong triangle inequality yields
\begin{align} \label{cjl f}
|c_{k,i,j_i,\lambda}|_{v_0}\le
|\mu_{\lambda}(\omega_i)|^{-1}_{v_0}\lambda^{\tfrac{(r-1)[K_{v_0}:\Q_{p_{v_0}}]}{[K:\Q]}}\enspace.
\end{align}
Finally, the identity $r_{k,\lambda}=\sum_{i=1}^m\sum_{j_i=0}^{r_i-1}c_{k,i,j_i,\lambda}$,
inequalities $(\ref{cjl inf})$ and $(\ref{cjl f})$ yield
\begin{align} \label{rkl}
|r_{k,\lambda}|_{v_0}\le 
\begin{cases}
e^{o(n)}\cdot \left|2^{n+1}\right|_{v_0}\cdot \left|\dfrac{2^nD_{n+1}}{(n+1)!}\right|^{r_1+\cdots +r_m}_{v_0}\cdot \lambda^{\tfrac{(r-1)[K_{v_0}:\R]}{[K:\Q]}+1} & \ \ \text{if} \ \ v_0\in \mathfrak{M}^{\infty}_K\\
\max_{1\le i \le m} (|\mu_{\lambda}(\omega_i)|^{-1}_{v_0})\lambda^{\tfrac{(r-1)[K_{v_0}:\Q_{p_{v_0}}]}{[K:\Q]}} & \ \ \text{if} \ \ v_0\in \mathfrak{M}^{f}_K\enspace.
\end{cases}
\end{align}
Using this inequality, we finish the proof of this proposition.

\
 
In the case of $v_0\in \mathfrak{M}^{\infty}_K$, Equation $(\ref{rkl})$ leads us to get
\begin{align*}
|D_{n+1}R_k(\alpha^{-1})|_{v_0}&\le e^{o(n)}\cdot \left|2^{n+1}\right|_{v_0}\cdot \left|\dfrac{2^nD_{n+1}}{(n+1)!}\right|^{r_1+\cdots +r_m}_{v_0}\cdot |\alpha^{-1}|^{(n+1)(r_1+\cdots r_m)}_{v_0}\cdot \\
&\sum_{\lambda=(n+1)(r_1+\cdots +r_m)+k-1}\lambda^{\tfrac{(r-1)[K_{v_0}:\R]}{[K:\Q]}+1}  \cdot |\alpha^{-1}|^{\lambda-(n+1)(r_1+\cdots +r_m)-k+1}_{v_0}\enspace.
\end{align*}
Here, since $|\alpha^{-1}|_{v_0}<1$, 
$$\sum_{\lambda=(n+1)(r_1+\cdots +r_m)+k-1}\lambda^{\tfrac{(r-1)[K_{v_0}:\R]}{[K:\Q]}+1}  \cdot |\alpha^{-1}|^{\lambda-(n+1)(r_1+\cdots +r_m)-k+1}_{v_0}=e^{o(n)}\enspace.$$
Combining the above inequalities allows us to obtain the desire estimate.

\

Next we consider the case of $v_0\in \mathfrak{M}^{f}_K$.
The strong triangle inequality leads us to obtain $$|D_{n+1}R_k(\alpha^{-1})|_{v_0}\le \max_{\lambda\ge k+(n+1)(r_1+\cdots+r_m)-1}(|r_{k,\lambda}\alpha^{-\lambda}|_{v_0})\enspace.$$
Using $$\max_{1\le i \le m} (|\mu_{\lambda}(\omega_i)|^{-1}_{v_0})\le 
\begin{cases} 
1 & \ \ \text{if} \ p_{v_0}\nmid {\rm{den}}(\omega_1,\ldots,\omega_m)\\
|p_{v_0}|^{-\tfrac{\lambda p_{v_0}} {p_{v_0}-1}}_{v_0}  & \ \ \text{if} \ p_{v_0}\mid {\rm{den}}(\omega_1,\ldots,\omega_m)\enspace,
\end{cases}
$$
By Equation $(\ref{rkl})$ and assumption $(\ref{cond a})$,
\begin{align*}
\max_{\lambda\ge k+(n+1)(r_1+\cdots+r_m)-1}(|r_{k,\lambda}\alpha^{-\lambda}|_{v_0}) & \le e^{o(n)}|\alpha^{-1}|^{(n+1)(r_1+\cdots+r_m)}_{v_0}\\
&\cdot \begin{cases} 
1 & \ \ \text{if} \ p_{v_0}\nmid {\rm{den}}(\omega_1,\ldots,\omega_m)\\
|p_{v_0}|^{-\tfrac{(r_1+\cdots+r_m)(n+1) p_{v_0}} {p_{v_0}-1}}_{v_0} & \ \ \text{if} \ p_{v_0}\mid {\rm{den}}(\omega_1,\ldots,\omega_m)\enspace,
\end{cases}
\end{align*}
holds. This completes the proof of Proposition $\ref{reminder term}$.
\end{proof}

\section{Proof of main theorem}
In this section, we conclude the proof of Theorem $\ref{power of log indep}$. 
Let us describe a sufficient condition for the non-vanishing of the determinant of matrices whose entries are consists of Pad\'{e} approximants for a given family of power series. 
\subsection{Non-vanishing of certain determinants}
\begin{lemma}  $(${\it confer}  {\rm{\cite[Theorem $1.2.3$]{J}}} $)$ \label{non vanishing of det}
Let $K$ be a field of characteristic $0$, $m$ an integer with $2\le m$ and $\boldsymbol{f}=(f_1,\ldots,f_{m})\in K[[z]]^m$ with $f_1(0)\neq 0$.
Let $\boldsymbol{n}=(n_1,\ldots,n_m)\in \Z^{m}_{\ge 0}$. Put $\boldsymbol{n}_k=(n_1+1,n_2+1,\ldots,n_k+1,n_{k+1},\ldots,n_m)$ for $1\le k \le m.$
Let $(P_{k,1}(z),\ldots,P_{k,m}(z))\in K[z]^m$ be a weight $\boldsymbol{n}_k$ Pad\'{e} approximant of $\boldsymbol{f}$. We define a polynomial $\Delta(z)$ by 
\begin{equation} \label{det 1}
                     \Delta(z)={\rm{det}}{\begin{pmatrix}
                     P_{1,1}(z)& P_{1,2}(z) & \dots &P_{1,m}(z)\\
                     P_{2,1}(z)& P_{2,2}(z) & \dots &P_{2,m}(z)\\
                     \vdots & \vdots & \ddots  &\vdots \\
                     P_{m,1}(z)& P_{m,2}(z) & \dots &P_{m,m}(z)\\
                     \end{pmatrix}}\enspace. 
                     \end{equation}
Then there exists an element $\gamma\in K$ satisfying 
\begin{align*} 
\Delta(z)=\gamma z^N,
\end{align*} 
where $N={\sum_{j=1}^m} (n_j+1)$.
Moreover, if the set of indices $\{\boldsymbol{n}\}\cup \{\boldsymbol{n}_k\}_{1\le k \le m-1}$ are normal with respect to $\boldsymbol{f}$, $\Delta(z)\neq 0$, i.e. $\gamma\neq 0$.
\end{lemma} 
\begin{proof}
Denote the formal power series ${\sum_{j=1}^m}P_{k,j}(z)f_j(z)$ by $R_k(z)$ for $1\le k \le m$. Then
\begin{align} \label{order lower bound}
{\rm{ord}}\, R_k(z)\ge N+k-1 \ \ \text{for} \ \ 1\le k \le m\enspace.
\end{align}
By adding the $j$-th column of the matrix in $(\ref{det 1})$ multiplied by $f_{j}(z)$ to the first column of the matrix for all $1\le k \le m$, we obtain 
$$
\Delta(z)=\dfrac{1}{f_1(z)}
                     \mathrm{det}{\begin{pmatrix}
                     R_1(z)& P_{1,2}(z) &\dots & P_{1,m}(z)\\
                     R_2(z)& P_{2,2}(z) & \dots & P_{2,m}(z)\\
                     \vdots & \vdots & \ddots & \vdots\\
                     R_m(z)& P_{m,2}(z) & \dots & P_{m,m}(z)\\
                     \end{pmatrix}}\enspace.                    
                     $$ 
For $1\le t,u\le m$, we denote the $(t,u)$-th cofactor of the above matrix by $\Delta_{t,u}(z)$. Then
\begin{align*} 
\Delta(z)=\dfrac{1}{f_1(z)}\sum_{t=1}^{m}R_{t}(z)\Delta_{t,1}(z)\enspace. 
\end{align*} 
Using inequalities $(\ref{order lower bound})$ and the above equality,
\begin{align*} 
{\rm{ord}}\, \Delta(z)\ge N\enspace.
\end{align*}
However, by the definition of  $\Delta(z)$, 
$$
{\rm{deg}}\, \Delta(z)\le N\enspace.
$$
These inequalities imply that the determinant $\Delta(z)$ is necessarily divisible by $z^{N}$.
Finally, assume the set of indices $\{\boldsymbol{n}\}\cup \{\boldsymbol{n}_k\}_{1\le k \le m-1}$ is normal with respect to $\boldsymbol{f}$. 
Thanks to Lemma $\ref{cor fund Pade}$, ${\rm{deg}}\,P_{k,k}(z)=n_k+1$ for $1\le k \le m$, and therefore, we conclude ${\rm{deg}}\, \Delta_{1,1}(z)=N$.
The inequalities 
$${\rm{ord}}\, R_t(z)=N+t-1, \ \ \ {\rm{deg}}\,\Delta_{t,1}(z)\le N \ \ \ \text{for} \ \ \ 1\le t \le m$$ 
enable us to obtain ${\rm{deg}}\, \Delta(z)=N$, in particular, $\Delta(z)\neq 0$.
\end{proof}
Using Lemma $\ref{non vanishing of det}$ for $\boldsymbol{f}=((1+z)^{\omega_i}{\rm{log}}^{j_i}(1+z))_{{1\le i \le m, 0 \le j_i \le r_i-1}}$ and $\boldsymbol{n}=(n,\ldots,n)\in \N^{r_1+\cdots+r_m}$, we obtain the following corollary.
\begin{corollary} \label{key corollary}
Let $(P_{k,i,j_i}(z))_{{1\le i \le m, 0\le j_i \le r_i-1}}$ be the vector of polynomials defined in Corollary $\ref{main pade}$ for $1\le k \le r_1+\cdots+r_m$. 
Put $\Delta(z)={\rm{det}}(P_{k,i,j_i}(z))_{\substack{1\le k \le r_1+\cdots+r_m \\ 1\le i \le m,0\le j_i \le r_i-1}}.$
There exists an element $\gamma\in K\setminus\{0\}$ with $\Delta(z)=\gamma z^{(n+1)(r_1+\cdots+r_m)}$. 
In particular we have $\Delta(\alpha)\neq 0$ for $\alpha\in K\setminus\{0\}$.
\end{corollary}
\begin{proof}
Put $\boldsymbol{n}=(n,\ldots,n)\in \N^{r_1+\cdots+r_m}$, $\boldsymbol{n}_k=(\underbrace{n+1,\ldots,n+1}_k,n,\ldots,n)\in \N^{r_1+\cdots+r_m}$ and 
$\boldsymbol{f}=((1+z)^{\omega_i}{\rm{log}}^{j_i}(1+z))_{{1\le i \le m, 0 \le j_i \le r_i-1}}$. 
By Proposition $\ref{diagonal normality log}$, the indices $\{\boldsymbol{n}\}\cup \{\boldsymbol{n}_k\}_{1\le k \le r_1+\cdots+r_m}$ are normal with respect to 
$\boldsymbol{f}$. 
Applying Lemma $\ref{non vanishing of det}$ for the weight $\boldsymbol{n}_k$ Pad\'{e} approximants $(P_{k,i,j_i}(z))_{{1\le i \le m, 0\le j_i \le r_i-1}}$ of
$\boldsymbol{f}$, 
we obtain the assertion.
\end{proof}
\subsection{A linear independence criterion}
Let us provide a criterion for obtaining a measure of linear independence of given numbers.
A similar criterion for simultaneous approximants for given numbers has been established in \cite[Proposition $3$]{DHK3}.
\begin{lemma} \label{criterion}
Let $K$ be an algebraic number field and $v_0$ a place of $K$.
Let $m\in\N$ with $m\ge 2$ and $\boldsymbol{\theta}=(\theta_1,\ldots,\theta_m)\in K_{v_0}^{{{m}}}$ with $\theta_1\neq0$. 
Suppose there exist a sequence of matrices $$({{{M}}}_n)_n=\left(\left(a^{(n)}_{l,j}\right)_{1\le l,j\le m}\right)_{n}\in  {\rm{GL}}_{m}(K)^{\N}\enspace,$$ 
real number $\mathbb{A}$ and functions $\{F_v:\N\rightarrow \R_{\ge0} \}_{v\in {{\mathfrak{M}}}_K}$ with
\begin{align}
&\max_{\substack{1\le l \le m}}\left\vert \sum_{j=1}^m a^{(n)}_{l,j}\theta_j\right\vert_{v_0}\leq e^{-\mathbb{A}n+o(n)}\enspace, \label{upper Rrj}\\
&\left\Vert {{{M}}}_n\right\Vert_v \leq  e^{F_v(n)} \ \ \text{for} \ \ v\in {{\mathfrak{M}}}_K\enspace, \label{upper Anrj}
\end{align}
for $n\in \N$.  
We assume there exist \begin{align*}
&\mathbb{B}={\limsup_{n\to \infty}}\dfrac{1}{n}\left((m-1){{\sum_{v\in \mathfrak{M}_K}}} F_v(n)-F_{v_0}(n)\right), \ \ U=\limsup_{n\to \infty}\dfrac{F_{v_0}(n)}{n}\enspace.
\end{align*}
Put $V=\mathbb{A}-\mathbb{B}$ and assume $V>0$. 
Then for any $0<\varepsilon<V$, there exists an effective constant $H_0=H_0(\varepsilon)>0$ depending on $\varepsilon$ and the given data such that the following property holds.
Then, for any ${{\boldsymbol{\beta}=(\beta_1,\ldots,\beta_m)}} \in K^{m} \setminus \{ \bold{0} \}$ satisfying $H_0\le {\rm{H}}({{\boldsymbol{\beta}}})$, 
\begin{align*}
\left|\sum_{j=1}^m{{\beta_j}}\theta_j\right|_{v_0}>C(\varepsilon) H_{v_0}(\boldsymbol{\beta}) {\rm{H}}({{\boldsymbol{\beta}}})^{-\mu(\varepsilon)}\enspace,
\end{align*}
where 
$$
\mu(\varepsilon)=
\dfrac{\mathbb{A}+U}{V-\varepsilon}\enspace \mbox{ \it and }\hspace{7pt} C{{(\varepsilon)}}=\dfrac{1}{2}\exp\left[-\frac{(V-\varepsilon+\log(2))(\mathbb{B}+U+\varepsilon)}{V-\varepsilon}\right]\enspace.
$$
\end{lemma}
\begin{proof}
We may assume $\theta_1=1$.
Let $\boldsymbol{\beta}=(\beta_1,\ldots,\beta_m)\in K^{m}\setminus \{\bold{0} \}$. 
Define 
$\Lambda({{\boldsymbol{\beta}}},\boldsymbol{\theta})={\sum_{j=1}^m}\beta_j\theta_j$.
Since ${{{M}}}_n$ is invertible and ${{\boldsymbol{\beta}}}\neq \bold{0}$, there exists $1\le l_n \le m$ with the value 
$$B_{n}={\rm{det}}
                    {\begin{pmatrix}
                     a^{(n)}_{1,1}& a^{(n)}_{1,2} & \dots &a^{(n)}_{1,m}\\
                     \vdots & \vdots & \ddots  &\vdots \\
                     \beta_1& \beta_1 & \dots & \beta_m\\
                     \vdots & \vdots & \ddots  &\vdots \\
                     a^{(n)}_{m,1} & a^{(n)}_{m,2} & \dots & a^{(n)}_{m,m}\\
                     \end{pmatrix}}\enspace,
$$
does not vanish where the vector $\boldsymbol{\beta}$ is in position the $l_n$-th line.
Put the $(l,1)$-th cofactor of the above matrix by $B_{n,l,1}$ and $r^{(n)}_{l}={\sum_{j=1}^m}a^{(n)}_{l,j}\theta_j$ for $1\le l \le m$.
The properties of determinant yields $$B_n=\sum_{l\neq l_n} r^{(n)}_l B_{n,l,1}+\Lambda(\boldsymbol{\beta},\boldsymbol{\theta})B_{n,l_n,1}\enspace.$$
Applying the product formula to $B_{n}\in K\setminus\{0\}$, 
\begin{equation} \label{upper infty}
1\le \prod_{v\neq v_0}\left\vert B_{n}\right\vert_v\cdot \left(\left\vert \sum_{l\neq l_n} r^{(n)}_l B_{n,l,1}\right\vert_{v_0}+\left\vert\Lambda(\boldsymbol{\beta},\boldsymbol{\theta})B_{n,l_n,1} \right\vert_{v_0}\right)\enspace.
\end{equation}
We first evaluate using inequalities $(\ref{upper Rrj})$ and $(\ref{upper Anrj})$
\begin{align*} 
&\log\left\vert B_{n} \right\vert_{v}\leq {\rm{h}}_v({{\boldsymbol{\beta}}})+(m-1){{F_v}}(n)+{\varepsilon_v \log\, (m!)}\ \ \text{for} \ \ v\in {{\mathfrak{M}}}_K\enspace,\\
&\log\left\vert \sum_{l\neq l_n} r^{(n)}_l B_{n,l,1}\right\vert_{v_0}\le -\mathbb{A}n+(m-2)F_{v_0}(n)+{\rm{h}}_{v_0}(\boldsymbol{\beta})+o(n)\enspace,\\
&\log\left\vert \Lambda(\boldsymbol{\beta},\boldsymbol{\theta})B_{n,l_n,1}\right\vert_{v_0}\le \log |\Lambda(\boldsymbol{\beta},\boldsymbol{\theta})|_{v_0}+(m-1)F_{v_0}(n)+\varepsilon_{v_0}\log(m-1)!\enspace.
\end{align*}
Now we fix a positive number $\varepsilon$ with $V>\varepsilon$, and assume that $n$ is large enough so that 
$\left\vert o(n)\right\vert\leq\varepsilon n/4$, $[K:\Q]\cdot {\rm{log}}(m!)\le \varepsilon n /4$, $0<-\log(2)+n(V-\varepsilon)$ and 
\begin{align*}
&\left\vert \left((m-1){{\sum_{v\in \mathfrak{M}_K}}} F_v(n)-F_{v_0}(n)\right)-n\mathbb{B}\right\vert\leq \varepsilon n/4, \ \ \ \ \left|nU-F_{v_0}(n)\right|\le \varepsilon n/4\enspace.
\end{align*}
Fix such $n^*$ and take $h_0=-\log(2)+n^*(V-\varepsilon)$ and $\boldsymbol{\beta}\in K^{m+1}\setminus \{\bold{0}\}$ with ${\rm{h}}(\boldsymbol{\beta})\ge h_0$. 
Let $n=n(\boldsymbol{\beta})$ be the minimal positive integer such that
${\rm{h}}({{\boldsymbol{\beta}}})-n(V-\varepsilon)\leq -{\rm{log}}(2)$. 
Note that $n\ge n^*$ and 
\begin{align} \label{ineq n}
n<\dfrac{{\rm{h}}({{\boldsymbol{\beta}}})}{V-\varepsilon}+\dfrac{\log(2)}{V-\varepsilon}+1\enspace.
\end{align}
With these conventions,
\begin{align*}
\sum_{v\neq v_0}\log\vert B_{n}\vert_v+\log\left\vert \sum_{l\neq l_n} r^{(n)}_l B_{n,l,1}\right\vert_{v_0}&\le-\mathbb{A}n+(m-1)\sum_{v\in\mathfrak{M}_K}F_v(n)-F_{v_0}(n)+{\rm{h}}(\boldsymbol{\beta})+\varepsilon n/2\\
&\le n(-\mathbb{A}+\mathbb{B}+\varepsilon)+{\rm{h}}(\boldsymbol{\beta})\\
&=-n(V-\varepsilon)+{\rm{h}}(\boldsymbol{\beta})<-\log(2) \enspace,
\end{align*}
and 
\begin{align*}
\sum_{v\neq v_0}\log\vert B_{n}\vert_v+\log\left\vert\Lambda(\boldsymbol{\beta},\boldsymbol{\theta})B_{n,l_n,1} \right\vert_{v_0}&\le \log|\Lambda(\boldsymbol{\beta},\boldsymbol{\theta})|_{v_0}+(m-1)\sum_{v\in \mathfrak{M}_K}F_v(n)+{\rm{h}}(\boldsymbol{\beta})-{\rm{h}}_{v_0}(\boldsymbol{\beta})+\varepsilon n/2\\
&\le \log|\Lambda(\boldsymbol{\beta},\boldsymbol{\theta})|_{v_0}+n(\mathbb{B}+U+\varepsilon)+{\rm{h}}(\boldsymbol{\beta})-{\rm{h}}_{v_0}(\boldsymbol{\beta})\enspace.
\end{align*}
Plugging in this inequality into Equation \eqref{upper infty},
\begin{align*}
-\log(2)-n(\mathbb{B}+U+\varepsilon)-{\rm{h}}(\boldsymbol{\beta})+{\rm{h}}_{v_0}(\boldsymbol{\beta})\le \log|\Lambda(\boldsymbol{\beta},\boldsymbol{\theta})|_{v_0}\enspace.
\end{align*}
Using Equation $(\ref{ineq n})$,
\begin{align*}
\log|\Lambda(\boldsymbol{\beta},\boldsymbol{\theta})|_{v_0}&\ge -\log(2)-\left(\dfrac{{\rm{h}}(\boldsymbol{\beta})+\log(2)}{V-\varepsilon}+1\right)(\mathbb{B}+U+\varepsilon)-{\rm{h}}(\boldsymbol{\beta})+{\rm{h}}_{v_0}(\boldsymbol{\beta})\\
&=\log C(\varepsilon)-\mu(\varepsilon){\rm{h}}(\boldsymbol{\beta})+{\rm{h}}_{v_0}(\boldsymbol{\beta})\enspace.
\end{align*}
This completes the proof of this lemma.
\end{proof}

\subsection{Proof of Main theorem}
We are ready to prove Theorem $\ref{power of log indep}$.

\begin{proof} [\textbf{Proof of Theorem $\ref{power of log indep}$}]
We use the same notation as in Theorem $\ref{power of log indep}$.
Let $\alpha\in K\setminus \{0,-1\}$.
For a non-negative integer $n$, we define matrix $${{M}}_n=(D_{n+1}\alpha^{n+1}P_{k,i,j_i}(\alpha^{-1}))_{\substack{1\le k \le {\sum_{i=1}^m}r_i \\ 1\le i \le m, 0\le j_i\le r_i-1}}\in {\rm{Mat}}_{r_1+\cdots+r_m}(K)\enspace,$$
where $D_{n+1}$ is the positive integer defined in $(\ref{E_n})$. 
By Corollary $\ref{key corollary}$, the matrices ${{M}}_n$ are invertible for every $n$. 
We define a family of functions 
\begin{align*}
F_v:\N\longrightarrow \R_{\ge0}; \ \ n\mapsto  &(n+1){\rm{h}}_v(\alpha,\alpha+1)+\varepsilon_v o(n)+
\dfrac{\varepsilon_v(r_1+\cdots+r_m)[K_v:\R]n}{[K:\Q]}\log (2)\\
&+\varepsilon_v r \log\left|\left[\prod_{1\le i_1<i_2\le m} D_{n+2}(\omega_{i_2}-\omega_{i_1})^2d_{n+2}(\omega_{i_2}-\omega_{i_1})\right]d_{n+1}\right|_v\enspace,
\end{align*} 
for $v\in \mathfrak{M}_K$.
For $\omega\in \Q\setminus \Z_{<0}$, by \cite{BKS}, 
$$
\limsup_{n\to \infty}\dfrac{1}{n}\log\, \max\{D_n(\omega), d_n(\omega)\}\le \dfrac{{\rm{den}}(\omega)}{\varphi({\rm{den}}(\omega))}\sum_{\substack{1\le j \le {\rm{den}}(\omega) \\ (j,{\rm{den}}(\omega))=1}}\dfrac{1}{j}\enspace,
$$
where $\varphi$ is Euler's totient function.
Therefore,
\begin{align*}
&\limsup_{n\to \infty}\dfrac{1}{n}\left( (r_1+\cdots+r_m-1)\sum_{v\in \mathfrak{M}_K}F_v(n)-F_{v_0}(n)\right)=\mathbb{B}_{v_0}(\boldsymbol{\omega},\alpha)\enspace,\\
&\limsup_{n\to \infty}\dfrac{F_{v_0}(n)}{n}=U_{v_0}(\boldsymbol{\omega},\alpha)\enspace,
\end{align*}
and, by Lemma $\ref{absolute value P}$,
\begin{align*}
\max_{\substack{1\le k \le r_1+\cdots+r_m \\ 1\le i \le m, 1\le j_i \le r_i-1}}\log|D_{n+1}\alpha^{n+1}P_{k,i,j_i}(\alpha^{-1})|_v\le F_{v}(n) \ \ \text{for} \ \ v\in \mathfrak{M}_K\enspace.
\end{align*}
Using Lemma $\ref{reminder term}$, 
\begin{align*}
\max_{1\le k \le r_1+\cdots+r_m} \log|D_{n+1}\alpha^{n+1}R_{k}(\alpha^{-1})|_{v_0}\le -\mathbb{A}_{v_0}(\boldsymbol{\omega},\alpha)n+o(n) \enspace.
\end{align*}
Applying Lemma $\ref{criterion}$ for the family of invertible matrices $({{M}}_n)_n$ and using above estimates, we obtain Theorem $\ref{power of log indep}$.
\end{proof}

\section{Appendix: perfectness of polylogarithms}
\subsection{Pad\'{e} approximants of Laurent series}
In this subsection, let us introduce the Pad\'{e} approximation of a Laurent series which is slightly different concept from the Pad\'{e} approximations of power series considered in Section $3$.
Throughout this subsection, we denote by $K$ a field of characteristic $0$.
We denote the formal power series ring of variable $1/z$ with coefficients $K$ by $K[[1/z]]$ and the field of fractions by $K((1/z))$. We say an element of $K((1/z))$ is a formal Laurent series.
Let $f(z)=\sum_{k=0}^{\infty}f_k/z^{k+1}\in (1/z)\cdot K[[1/z]]$ and $n$ a non-negative integer.
We define a $K$-homomorphism $\varphi_f$ associated with $f$ by
$$\varphi_f:K[t]\longrightarrow K; \ \ t^k\mapsto f_k \ \ (k\ge 0)\enspace.$$ 
\begin{definition} \label{Pade infty}
Let $(P,Q)\in K[z]^2$ be a pair of polynomials. 
We say the pair $(P,Q)$ is a weight $n$ Pad\'{e} approximants of $f(z)$ if $(P,Q)$ satisfies
\begin{align*}
&(i) \ \ P(z)\neq 0\enspace,\\
&(ii) \ \  {\rm{deg}}\, P\le n\enspace,\\
&(iii) \ \ {\rm{ord}}_{\infty} \, P(z)f(z)-Q(z)\ge n+1\enspace.
\end{align*}
We denote the set of weight $n$ Pad\'{e} approximants of $f$ by ${\rm{PA}}_n(f)$.
Let $(P,Q)\in {\rm{PA}}_n(f)$. 
Note that the polynomial $Q$ is uniquely determined such as $Q(z)=\varphi_f(\tfrac{P(z)-P(t)}{z-t})$.
We denote the reminder function $P(z)f(z)-Q(z)$ of $(P,Q)\in {\rm{PA}}_n(f)$ by $R_P(z)$.
\end{definition}
\begin{definition}
Let $n$ be a non-negative integer. 
We say $n$ is normal with respect to $f$ if and only if, for any $(P,Q)\in {\rm{PA}}_n(f)$, we have ${\rm{ord}}_{\infty}\,R_P=n+1$.
We denote the set of non-negative integers which are normal with respect to $f$ by $\Lambda(f)$.
\end{definition}
\begin{remark}
Let us describe the relationship between Pad\'{e} approximants defined in Definition $\ref{Pade infty}$ and that in Definition $\ref{Pade}$.
Let $f(z)\in (1/z)\cdot K[[1/z]]$. 
To give a pair in ${\rm{PA}}_n(f)$ is equivalent to give a weight $(n,n-1)$ Pad\'{e} approximants of $(1/z\cdot f(1/z),1)$. 
More precisely, the map 
\begin{align*}
& \ \ \ \ \ {\rm{PA}}_n(f)\longrightarrow \{(P_1,P_2)\in K[z]^2\mid (P_1,P_2) \ \text{is a weight} \ (n,n-1) \ \text{Pad\'{e} approximants of} \ (1/z\cdot f(1/z),1)\};\\
& (P(z),Q(z))\mapsto (z^nP(1/z),-z^{n-1}Q(1/z))\enspace,
\end{align*}
is bijective. This yields that a non-negative integer $n$ is normal with respect to $f$ is equivalent to the pair $(n,n-1)$ is normal with respect to $(1/z\cdot f(1/z),1)$ in the sense of Definition $\ref{Pade 0}$.
\end{remark}
\begin{definition}
Let $m,n$ be non-negative integers. Define the $(m,n)$-Hankel matrix of $f$ by
$$H_{m,n}(f)=
\begin{pmatrix} 
f_0 & f_1 & \cdots & f_n\\
f_1 & f_2 & \ldots & f_{n+1}\\
\vdots & \vdots & \ddots & \vdots\\
f_{m-1} & f_{m} & \ldots & f_{m+n}
\end{pmatrix}\in {\rm{Mat}}_{m,n+1}(K)\enspace.
$$
\end{definition}
The similar argument to that in Remark $\ref{remark bij}$ leads us to get the following equivalence relations. 
\begin{lemma}\label{equiv}
Let $P(z)=\sum_{k=0}^np_{k}z^k\in K[z]$ be a nonzero polynomial. Put $Q(z)=\varphi_f(\tfrac{P(z)-P(t)}{z-t})$.
For a non-negative integer $n$, the following are equivalent.

$(i)$ We have $(P,Q)\in {\rm{PA}}_n(f)$.

$(ii)$ The vector ${}^t(p_0,\ldots,p_n)\in K^{n}$ is contained in the kernel of the $K$-linear map $$H_{n,n}(f):K^{n+1}\longrightarrow K^{n}; \ \ \vec{x}\mapsto H_{n,n}(f)\cdot \vec{x}\enspace.$$
\end{lemma}
The similar argument to the proof of Lemma \ref{cor fund Pade} leads us to conclude that the given non-negative integer $n$ being normal with respect to $f(z)$ ({\it confer} \cite{N-S}).
\begin{lemma}\label{equiv2}
For $f(z)\in (1/z)\cdot K[[1/z]]$ and a non-negative integer $n$ the following are equivalent.

$(i)$ $n\in \Lambda(f)$. 

$(ii)$ ${\rm{det}}\,H_{n+1,n}(f)\neq 0$.
\end{lemma}

\subsection{Perfectness of polylogarithm function}
Let $r$ be a positive integer. We denote the $r$-th polylogarithm function by ${\rm{Li}}_r(z)$, namely
$${\rm{Li}}_r(z)=\sum_{k=0}^{\infty}\dfrac{z^{k+1}}{(k+1)^r}\in \Q[[z]]\enspace.$$
In this subsection, we prove the perfectness of ${\rm{Li}}_r(1/z)$.
Denote the $(n,n-1)$-Hankel matrix of ${\rm{Li}}_r(1/z)$ by
$$H^{(r)}_n=
\begin{pmatrix}  
1 & 1/2^r & \cdots & 1/n^r\\
1/2^r & 1/3^r & \ldots & 1/(n+1)^r\\
\vdots & \vdots & \ddots & \vdots\\
1/n^r & 1/(n+1)^r & \ldots & 1/(2n)^r
\end{pmatrix}\enspace.
$$
\begin{proposition} \label{non zero}
${\rm{det}}\,H^{(r)}_{n}\neq 0.$
\end{proposition}
\begin{proof} Let $t_1,\ldots,t_n$ be variables. 
Define $\Q$-linear maps: 
$$\varphi^{(i)}_r:\Q[t_i]\longrightarrow \Q; \ \ t^k_i\mapsto \dfrac{1}{(r-1)!}\int^1_0t^k_i\log^{r-1}\left(\dfrac{1}{t_i}\right)dt_i=\dfrac{1}{(k+1)^r} \ \ \ (1\le i \le n)\enspace.$$
According to the definition of $\varphi^{(i)}_r$, 
\begin{align*}
{\rm{det}}\,H^{(r)}_{n}&=\begin{vmatrix} 
\varphi^{(1)}_r(1) & \varphi^{(2)}_r(t_2) & \cdots & \varphi^{(n)}_r(t^{n-1}_n)\\
\varphi^{(1)}_r(t_1) & \varphi^{(2)}_r(t^2_2) & \cdots & \varphi^{(n)}_r(t^{n}_n)\\
\vdots & \vdots & \ddots & \vdots\\
\varphi^{(1)}_r(t^{n-1}_1) & \varphi^{(2)}_r(t^n_2) & \cdots & \varphi^{(n)}_r(t^{2n-1}_n)\\
\end{vmatrix}
=\varphi^{(1)}_r\circ\cdots \circ \varphi^{(n)}_r\left(
\begin{vmatrix} 
1 & t_2 & \cdots & t^{n-1}_n\\
t_1 & t^2_2 & \cdots & t^{n}_n\\
\vdots & \vdots & \ddots & \vdots\\
t^{n-1}_1 & t^n_2 & \cdots & t^{2n-1}_n\\
\end{vmatrix}\right)\\
&=\varphi^{(1)}_r\circ\cdots \circ \varphi^{(n)}_r\left(t_2t^2_3\cdots t^{n-1}_n\prod_{1\leq i<j\leq n}(t_{j}-t_{i})\right)\\
&=\dfrac{1}{[(r-1)!]^n}\int_{[0,1]^n}t_2t^2_3\cdots t^{n-1}_n\prod_{1\leq i<j\leq n}(t_{j}-t_{i})\prod_{i=1}^n\log\left(\dfrac{1}{t_i}\right)dt_i\enspace.
\end{align*}
We denote the $n$-th symmetric group by $\mathfrak{S}_n$ and put
\begin{align*}
&D_{\sigma}=\{(t_1,\ldots ,t_n)\in [0,1]^n \mid t_{\sigma(1)}\le \cdots \le t_{\sigma(n)}\} \ \text{for} \ \sigma\in \mathfrak{S}_n\enspace.
\end{align*}
Denoting $g(\boldsymbol{t})=t_2t^2_3\cdots t^{n-1}_n\prod_{1\leq i<j\leq n}(t_{j}-t_{i})\prod_{i=1}^n\log\left(1/t_i\right)$ and $d\mu$ of Lebesgue measure on $[0,1]^n$ 
and using  the change of variables
$\sigma^{-1}:\kern5pt
D_{{\rm{id}}} \longrightarrow D_{\sigma},  \ (t_1,\ldots,t_n) \mapsto (t_{\sigma^{-1}(n)}, \ldots, t_{\sigma^{-1}(n)})\enspace,
$
$$
\int_{[0,1]^n} g d\mu=\int_{D_{{\rm{id}}}}\sum_{\sigma\in \mathfrak{S}_n} g\circ\sigma^{-1} d\mu \enspace.
$$
Combining $\prod_{1\leq i<j\leq n}(t_{\sigma^{-1}(j)}-t_{\sigma^{-1}(i)})={\rm{sgn}}(\sigma)\prod_{1\leq i<j\leq n}(t_{j}-t_{i})$ and remark (Vandermonde)
yields 
\begin{align*}
{\rm{det}}\,H^{(r)}_{n}&=\dfrac{1}{[(r-1)!]^n}\int_{D_{{\rm{id}}}}\sum_{\sigma\in \mathfrak{S}_n}{\rm{sgn}}(\sigma) t^0_{\sigma(1)}t_{\sigma(2)}\cdots t^{n-1}_{\sigma(n)}\cdot \prod_{1\leq i<j\leq n}(t_{j}-t_{i})\prod_{i=1}^n\log\left(\dfrac{1}{t_i}\right)dt_i\\
&=\dfrac{1}{[(r-1)!]^n}\int_{D_{{\rm{id}}}} \left[\prod_{1\leq i<j\leq n}(t_{j}-t_{i})\right]^2 \prod_{i=1}^n\log\left(\dfrac{1}{t_i}\right)dt_i >0\enspace.
\end{align*}
This completes the proof of the proposition.
\end{proof}
Combining Proposition $\ref{non zero}$ and Lemma $\ref{equiv2}$ yields the perfectness of ${\rm{Li}}_r(1/z)$. 

\

{\bf Acknowledgements.}

This work is partially supported by the Research Institute for Mathematical Sciences, an international joint usage and research center located at Kyoto University.
The author is supported by JSPS KAKENHI Grant Number JP24K16905.
The author deeply thanks N.~Hirata-Khono for her enlightening comments on a preliminary version.

\bibliography{}

\begin{thebibliography}{99}
\bibitem{B}
A.~Baker,
\textit{Approximations to the logarithms of certain rational numbers},
Acta Arith.\ \textbf{10} (1964), 315--323.


\bibitem{BKS}
P.~Bateman, J.~Kalb and A.~Stenger, 
{\it Problem $10797$: A limit involving least common multiples},
Amer. Math. Monthly \textbf{109} (2002), 393--394.

\bibitem{Beu}
F.~Beukers,
\textit{Irrationality of some $p$-adic $L$-values},
Acta Math.\ Sin. \textbf{24},\ no.\ 4, (2008), 663--686.

\bibitem{chpi}
G.~V.~Chudnovsky, {\it Hermite-Pad\'e approximations to exponential functions and elementary estimates of the measure of irrationality of $\pi$},
Lecture Notes in Math. \textbf{925}, 1982, 299--322.

\bibitem{ch9}
G.~V.~Chudnovsky,
{\it On the method of Thue-Siegel},
Ann. of Math. \textbf{117} (1983)  325--382.

\bibitem{ch1}
G.~V.~Chudnovsky,
{\it Rational approximations to linear forms of exponentials and binomials},
Proc. Natl. Acad. Sci. USA, \textbf{80} (1983), 3138--3141.

\bibitem{C}
P.~L.~Cijsouw,
\textit{Transcendence measures of exponentials and logarithms of algebraic numbers},
Compositio Math. \textbf{28} (1974), 163--178.

\bibitem{DHK3}
S.~David, N.~Hirata-Kohno  and M.~Kawashima,
{\it Linear Forms in Polylogarithms},
Ann. Sc. Norm. Super. Pisa Cl. Sci. (5)  \textbf{23} (2022), 1447--1490.

\bibitem{F}
N.~I.~Fel'dman,
\textit{Approximation of the logarithms of algebraic numbers by algebraic numbers},
English transl., Amer.\ Math.\ Soc.\ Transl.\ II ser., \textbf{58} (1966), 125--142.

\bibitem{FN}
N.~I.~Fel'dman and Yu.~V.~Nesterenko,
Number Theory IV (eds. A.~N.~Parshin and  I.~R.~Schfarevich),
Encyclopaedia of Mathematical Sciences, \textbf{44,} Springer, 1998.

\bibitem{G-G}
A.~Galochkin and A.~Godunova
{\it On approximations of solutions of the equation $P(z,\ln z)$ by algebraic numbers},
Mosc. J. Comb. Number Theory \textbf{9} (2020), 435--440.

\bibitem{G}
A.~O.~Gel'fond,
\textit{Transcendental and algebraic numbers},
English transl.\ (Dover Publications, New York (1960)).


\bibitem{H}
C.~Hermite,
\textit{Sur la fonction exponentielle},
Oeuvres tome III (1873), 150--181.

\bibitem{Huttner}
M.~Huttner,
\textit{Probl\'{e}me de Riemann effectif et approximants de Pad\'{e}-Hermite},
Approximations Diophantiennes et Nombres Transcendents, Luminy 1990, De Gruyter \text{and} \ Co., Berlin-New York 1992, 173--193.

\bibitem{J}
H.~Jager,
\textit{A multidimensional generalization of the Pad\'{e} table. I, I\hspace{-.1em}I, I\hspace{-.1em}I\hspace{-.1em}I, I\hspace{-.1em}V, V, V\hspace{-.1em}I},
Nederl.~Ak.~Wetenschappen, \textbf{67} (1964), 192--249.


\bibitem{L}
F.~Von Lindeman,
\textit{$\ddot{U}$ber die Zahl $\pi$},
Math.\ Ann.\ \textbf{20} (1882), 213--225.

\bibitem{M0}
K.~Mahler,
\textit{Zur Approximation der Exponentialfunktion und des Logarithmus. Teil I},
J.\ Reine.\ Angew.\ Math., \textbf{166} (1931), 118-136.

\bibitem{M}
K.~Mahler,
\textit{A proof of the Thue-Siegel Theorem about the approximation of algebraic numbers for binomial equations},
Math. Ann. \textbf{105} (1931), 267--276. 

\bibitem{M1}
K.~Mahler,
\textit{Zur Approximation der Exponentialfunktion und des Logarithmus. Teil I\hspace{-.1em}I},
J.\ Reine.\ Angew.\ Math., \textbf{166} (1932), 137--150.

\bibitem{M2}
K.~Mahler,
\textit{On the approximation of logarithms of algebraic numbers},
Phil.\ Trans.\ Royal Soc., \textbf{245} (1953), 371--398.

\bibitem{M3}
K.~Mahler,
\textit{Application of some formulae by Hermite to the approximation of exponentials and logarithms},
Math.\ Ann., \textbf{168} (1967), 200--227.

\bibitem{M4}
K.~Mahler,
\textit{Perfect systems},
Compositio Mathematica (1968) Volume: \textbf{19}, Issue: 2, 95--166.

\bibitem{Mig}
M.~Mignotte, 
\textit{Approximations rationnlles de $\pi$ et quelques autres nombres},
Journ\'{e}es Arithm\'{e}tiques (Grenoble, 1973), Soc. Math. France, Paris (1974) Bull. Soc. Math. France, M\'{e}m. \textbf{37}, 121--132. 

\bibitem{N-W}
Yu.~Nesterenko, M.~Waldschmidt,
\textit{On the approximation of the values of exponential function and logarithm by algebraic numbers}
Mat. Zapiski, \textbf{2}, Diophantine approximations, Proceedings of papers dedicated to the memory of Prof. N.I.~Feldman, ed. Yu.V.~Nesterenko, Centre for applied research under Mech.-Math.
Faculty of MSU, Moscow (1996), 23--42.

\bibitem{N-S}
E.~M.~Nikisin, V.~N.~Sorokin,
\textit{Rational Approximations and Orhogonality $($Translations of Mathematical Monographs$)$},
American Mathematical Society, (1991). 

\bibitem{R}
E.~Reyssat,
\textit{Mesures de transcendence pour les logarithmes de nombres rationnels},
Approximations diophantiennes et nombres transcendants, Luminy, 1982 Progress in Math 235--245, Birkh$\ddot{\text{a}}$user (1983).

\bibitem{R-S1}
J.~B.~Rosser, L.~Schoenfeld,
\textit{Approximate formulas for some functions of prime numbers},
Illinois J.\ Math.\ \textbf{6} (1962),  64--94.

\bibitem{R-S2}
J.~B.~Rosser, L.~Schoenfeld,
\textit{Shaper bounds for the Chebyshev functions $\theta(x)$ and $\psi(x)$},
Math.\ Comp.\ \textbf{29}, (1975), 243--269.


\bibitem{S1}
C.~L.~Siegel,
\textit{Transcendental Numbers}, 
Annals of Mathematics Studies \textbf{16}, Princeton University Press, 1949. 

\bibitem{S}
C.~L.~Siegel,
{\it $\ddot{\text{U}}$ber einige Anwendungen diophantischer Approximationen},
Abh. Preu\ss. Akad. Wiss., Phys.-Math. Kl. (1),  1929,  70S,
English transl. in {\it On Some Applications of Diophantine Approximations}, (with a commentary by C. Fuchs and U. Zannier),
Quad. Monogr. \textbf{2}, Edizioni della Normale, Pisa, 
2014, 1--80.

\bibitem{W1} 
M.~Waldschmidt,
\textit{Transcendence measures for exponentials and logarithms},
J.\ Austral.\ Math.\ Soc. (Series A) \textbf{25} (1978), 445--465.

\bibitem{W2}
M.~Waldschmidt,
\textit{Diophantine approximation on linear algebraic groups: Transcendence properties of the exponential function in several variables},
Grundlehren der mathematischen Wissenschaften, \textbf{326}. Springer-Verlag, Berlin Heidelberg, 2000.

\bibitem{Wei}
K.~Weierstrass,
\textit{Zu Lindemann's Abhandlung ``$\ddot{u}$ber die Ludolph'sche Zahl''},
Sitzungsberichte der K$\ddot{\text{o}}$niglich-Preu$\beta$ischen Akademie der Wissenschaften (1885), 1067--1085.

\end{thebibliography}

\medskip\vglue5pt
\vskip 0pt plus 1fill
\hbox{\vbox{\hbox{Makoto \textsc{Kawashima}}
\hbox{Institute for Mathematical Informatics, Meiji Gakuin University}
\hbox{1518 Kamikurata-chyo Totsuka-ku Yokohama Kanagawa}
\hbox{244-8539, Japan}
\hbox{{\tt kawashima\_makoto@mi.meijigakuin.ac.jp}}
}}


\end{document}